\documentclass[10pt]{article}

\usepackage[letterpaper,top=1in, bottom=1in, left=1in, right=1in]{geometry}
\usepackage[]{algorithm2e}
\usepackage{amsmath}
\usepackage{mathrsfs}
\usepackage{amssymb}
\usepackage{amsfonts}
\usepackage{gauss}
\usepackage{dsfont}
\usepackage{amsthm}
\usepackage{graphicx}
\usepackage{epstopdf}
\usepackage{enumerate}
\usepackage{listings}
\usepackage{multirow}
\usepackage{chemarrow}
\usepackage{extarrows}
\usepackage{bbm}
\usepackage{url}
\usepackage[numbers]{natbib} 
\usepackage{hyperref}
\usepackage{titling}
\usepackage[pagestyles]{titlesec}
\usepackage{color}
\usepackage{csquotes}
\usepackage{setspace}

\newtheoremstyle{mystyle}{2ex}{}{}{}{\bfseries}{.}{1ex}{}
\theoremstyle{mystyle} \newtheorem{claim}{Claim}
\theoremstyle{mystyle} 
\theoremstyle{mystyle} \newtheorem{lemma}{Lemma}
\theoremstyle{mystyle} 
\theoremstyle{mystyle} 
\theoremstyle{mystyle} 

\theoremstyle{mystyle} 
\theoremstyle{mystyle} \newtheorem{assump}{Assumption}

\newcommand{\E}{\mathbb{E}}

\newcommand{\Var}{\mathrm{Var}}

\newtheorem{thm}{Theorem}
\newtheorem{observation}{Observation}
\newtheorem{cor}{Corollary}
\newtheorem{defn}{Definition}
\theoremstyle{remark}
\newtheorem{remark}{Remark}
\newtheorem{example}{Example}

\usepackage{mathtools}

\mathtoolsset{showonlyrefs}
\newcommand{\T}{\lfloor c\log n\rfloor}

\begin{document}
\title{Balanced Allocation with Random Walk Based Sampling
	\thanks{The authors would like acknowledge support by NSF via grants AST-1343381, AST-1516075, IIS-1538827 and ECCS-1608361, and also Harsha Honappa and Yang Xiao for early discussions. 
	}}
\author{Dengwang Tang, Vijay G. Subramanian}
\date{}

\maketitle

\begin{abstract}
	In the standard ball-in-bins experiment, a well-known scheme is to sample $d$ bins independently and uniformly at random and put the ball into the least loaded bin. It can be shown that this scheme yields a maximum load of $\log\log n/\log d+O(1)$ with high probability. 
	
	Subsequent work analyzed the model when at each time, $d$ bins are sampled through some correlated or non-uniform way. However, the case when the sampling for different balls are correlated are rarely investigated. In this paper we propose three schemes for the ball-in-bins allocation problem. We assume that there is an underlying $k$-regular graph connecting the bins. The three schemes are variants of \emph{power-of-$d$ choices}, except that the sampling of $d$ bins at each time are based on the locations of $d$ independently moving non-backtracking random walkers, with the positions of the random walkers being reset when certain events occurs. We show that under some conditions for the underlying graph that can be summarized as the graph having large enough girth, all three schemes can perform as well as \emph{power-of-$d$}, so that the maximum load is bounded by $\log\log n/\log d+O(1)$ with high probability.

	
\end{abstract}
	
\section{Introduction}

The model of placing $n$ balls into $n$ bins has been analyzed for decades. Despite its simplicity, this model has a wide range of applications, such as analysis of load balancing in distributed hash tables and distributed load balancing in many server systems. A classical result tells us that if each ball is inserted into a uniform random bin, then the maximum load, i.e. the number of balls in the fullest bin, will be $(1+o(1))\log n/\log \log n$ with high probability \cite{raab1998balls}\cite{mitzenmacher2001power}. Azar et. al. \cite{azar1999balanced} then proposed the following scheme, called \emph{power-of-$d$ choices}: at the time of placing the $i$-th ball, pick $d$ bins \emph{independently and uniformly} at random, compare the load of $d$ bins, and insert the $i$-th ball into the least loaded bins (ties are broken arbitrarily). It was shown that, this scheme yields a maximum load of $\log \log n / \log d + O(1)$ with high probability \cite{azar1999balanced}\cite{mitzenmacher2001power}. This result implies that, with a small amount of random choices, the maximum load can be significantly reduced \cite{richa2001power}.

Many variants of the $d$-random-choice model have been proposed and analyzed. Most of these works focus on alternative ways to sample $d$ bins, either non-uniformly or non-independently. V\"{o}cking \cite{vocking2003asymmetry} analyzed a scheme where $d$ bins are sampled from $d$ disjoint groups of bins respectively, and ties are broken asymmetrically. It was shown that this scheme reduces the maximum load to $\log \log n/(d \log \phi_d)+O(1)$ with high probability, where $1<\phi_d<2$ \cite{vocking2003asymmetry}. Kenthapadi and Panigraphy \cite{kenthapadi2006balanced} analyzed the scheme where $d=2$ bins are sampled through picking a random edge of an underlying graph $G$, where $G$ is assumed to be almost $n^\varepsilon$ regular for some $0<\varepsilon<1$. It was shown that this scheme yields a maximum load of $\log\log n + O(1/\varepsilon)$ \cite{kenthapadi2006balanced}. As a comparison, Azar's model corresponds to the case that $G$ is a complete graph. Godfrey \cite{godfrey2008balls} then generalized the results to the case where an indefinite number of bins are sampled through randomly picking a subset $B_i\subset [n]$ according to some probability measure. It was shown that, when the size of the subset is approximately $\Theta(\log n)$, and for each bin $j\in [n]$, the probability that $j\in B_i$ is at the same order of $(\log n)/n$, the maximum load is $\Theta(1)$ with high probability. A salient feature of the scheme in \cite{godfrey2008balls} is that the subsets of bins can be arbitrarily correlated across balls, and only when $\Theta(\log n)$ bins get sampled is there sufficient spread in the selections. More specifically, the \emph{power-of-$d$ choices} scaling doesn't hold if the number of bins sampled is some finite number $d$. Sampling using a single random walk on a graph has also been utilized to probe bins in balls-in-bins models. Pourmiri \cite{pourmiri2016balanced} proposed and analyzed the scheme where the placement bin for ball $i$ is sampled using minimally loaded bins from a specific subset of locations visited by a \textit{single} non-backtracking random walk of length $l r_G=o(\log(n))$ (where $l=o(\log(n))$ and $r_G\sim \log\log(n)$) on a high girth $d$-regular graph between the bins, which \textit{starts from a uniformly random position} in the graph. It is shown for sparse $d$-regular graphs ($d\in [3, O(\log(n))]$) with high girth that when $l\geq 4\sqrt{\log n /\log k}$, the maximum load is $O(\log \log n/\log(l/\sqrt{\log n /\log k}))$ with high probability. Again a salient feature of \cite{pourmiri2016balanced} is the correlated sampling of bins for each balls, but with independence of the sampling set across balls.

In most of the works above, except for  \cite{godfrey2008balls}, despite the existence of correlation of bins sampled by the same ball, the sets of bins sampled by different balls are still independent. The case where correlations exists across sampling processes for different balls are rarely investigated, with the exception of  \cite{godfrey2008balls} where by sampling a large number of nodes, the dependence doesn't matter as much. Alon et al. \cite{alon2007non} investigated the maximum number of times a vertex $j$ is visited by a non-backtracking random walk of length $n$ on a $k$-regular graph $G$ with $n$-vertices. In the language of balls-in-bins experiment, this model is equivalent to inserting $n$ balls to $n$ bins based on one non-backtracking random walk of length $n$. It was shown that, when $G$ is an expander graph with high girth, the maximum load is $(1+o(1))\log n/\log \log n$ with high probability \cite{alon2007non}, which coincides with the result when each ball chooses a independent uniform random position. Alon et al. \cite{alon2007non} then raised the following open problem: 
\begin{displayquote}
	\textit{... Let $W_1$ and $W_2$ denote two non-backtracking random walks on  an expander of high girth, and suppose that in each step we are given a choice between the two current locations of $W_1$ and $W_2$, and pick the least loaded one. Does the maximal load decrease from $\Theta(\frac{\log n}{\log\log n})$ to $\Theta(\log\log n)$ in this setting as well? ...
	}
\end{displayquote}
This problem can also be considered in the setting in \cite{godfrey2008balls}: \textit{Can some dependent selection (across balls) of just a finite number of bins per ball still yield power-of-d choices performance?}

In the computer science literature, random walks on expander graphs are utilized to reduce the number of random bits used by a randomized algorithm to achieve a certain objective on error probability, or to reduce error probability under the same randomness budget. This procedure is referred to as \emph{derandomization}. The importance of reducing random bits lies in the fact that random bits, just like memory, are important resources in a computer. See \cite{hoory2006expander} for a survey on this topic.

\textbf{\textit{Overview of Results:}} Motivated by Alon's open problem and the literature on derandomization, we seek to modify the \emph{power-of-$d$ choices} scheme by replacing independent and uniformly random sampling with non-backtracking random walks based sampling. We assume that the bins are interconnected through a $k$-regular graph $G$ that remains fixed. The $d$ bins used at each time will be the positions of $d$ non-backtracking random walkers $W_1, W_2,\cdots, W_d$ on a graph $G$. The random walkers are periodically reset to independent uniformly random vertices. We describe and analyze three schemes. For the first scheme, the random walkers are reset if either two random walkers' path (since last reset) intersects, or a reset has not taken place for $\lfloor c\log n\rfloor$ timesteps. For the second scheme, the random walkers reset only every $\lfloor c\log n\rfloor$ timesteps. For the third scheme, we completely eliminate the resets. Under certain assumptions on the graph $G$, such as large girth, we show that both schemes can achieve the same performance as \emph{power-of-$d$ choices} scheme, that is, the maximum load of all bins is bounded by $\log\log n/\log d+\Theta(1)$ with high probability. Our result can be interpreted as a derandomization of the \emph{power-of-$d$ choices} scheme. Our proof uses the basic iterated bounding technique outlined in \cite{azar1999balanced} but with martingale concentration inequalities to account for the dependence. Finally, as uniformly random sampling is used in many other contexts such as peer-selection in peer-to-peer communications~\cite{hajek2011missing}, 
our broader goal is to study derandomization of such procedures in many other applications. 


\subsection{Large Deviation Inequalities}

We will use a corollary of Bernstein Inequality for martingales, which we shall cite here for completeness.

\begin{thm}[Bernstein]
	If a martingale $(\mathbf{X}, \mathcal{F})$ satisfies \cite{chung2006concentration}
	\begin{equation}
	\Var[X_i|\mathcal{F}_{i-1}] \leq \sigma_i^2~~~~~~\text{a.s. }
	\end{equation}
	for $1\leq i\leq n$, and
	\begin{equation}
	|X_i - \E[X_i|\mathcal{F}_{i-1}]| \leq B~~~~~~\text{a.s. }
	\end{equation}
	for $1\leq i\leq n$. Then we have
	\begin{equation}
	\Pr(X_n\geq X_0 + \lambda) \leq \exp\left(-\dfrac{\lambda^2}{2\left(\sum_{i=1}^n \sigma_i^2  + \frac{B\lambda }{3} \right)} \right)
	\end{equation}
\end{thm}

We will also use Markov Inequality!

\begin{cor}\label{berncor}
	Let $\{Z_j\}_{j=1}^\infty$ be an adapted process w.r.t. the filtration $\{\mathcal{F}_j\}_{j=1}^\infty$. If
	\begin{equation*}
	0\leq Z_j\leq B~~~~~~~~a.s.
	\end{equation*}
	and
	\begin{equation}
	\E[Z_j~|~\mathcal{F}_{j-1}]\leq m~~~~~~~~a.s.
	\end{equation}
	then for any $\lambda\geq 2Nm$, we have
	\begin{equation*}
	\Pr\left(\sum_{j=1}^{N} Z_j \geq \lambda\right)\leq \exp\left(-\dfrac{3\lambda}{16B}\right)
	\end{equation*}
	
\end{cor}

\begin{proof}
	The process $S_j = \sum_{l=1}^j (Z_j - \E[Z_j~|~\mathcal{F}_{j-1}])$ is a martingale with respect to $\{F_j\}_{j=1}^\infty$, and we have
	\begin{equation}\label{asbound}
	S_{N} = \sum_{l=1}^{N} Z_l - \sum_{l=1}^N \E[Z_l~|~\mathcal{F}_{l-1}] \geq \sum_{l=1}^{N} Z_l - Nm  ~~~~~~~~a.s.
	\end{equation}
	
	It is also clear that
	\begin{equation}
	|S_j-S_{j-1}|\leq B~~~~~~~~a.s.
	\end{equation}
	
	We can bound the conditional variances of the martingale $S_j$ through
	\begin{align*}
	\Var[S_j~|~\mathcal{F}_{j-1}] &= \Var[Z_j~|~\mathcal{F}_{j-1}] \leq \E[Z_j^2~|~\mathcal{F}_{j-1}]\\
	&\leq B\E[Z_j~|~\mathcal{F}_{j-1}] \leq Bm~~~~~~~~a.s.
	\end{align*}
	
	Thus, by \eqref{asbound} and Bernstein Inequality we have
	\begin{align*}
	\Pr\left(\sum_{j=1}^{N} Z_j \geq \lambda\right)&\leq \Pr\left(\sum_{j=1}^{N} Z_j - Nm \geq \dfrac{\lambda}{2}\right)\leq \Pr\left(S_{N} \geq \dfrac{\lambda}{2}\right)\\
	&\leq \exp\left(-\dfrac{(\frac{\lambda}{2})^2}{2\left[\sum_{j=1}^N Bm + \frac{B}{3}\frac{\lambda}{2}\right]}\right)=\exp\left(-\dfrac{(\frac{\lambda}{2})^2}{2\left[BNm + \frac{B}{3}\frac{\lambda}{2}\right]}\right)\\
	&\leq\exp\left(-\dfrac{\frac{1}{4}\lambda^2}{2\left[ \frac{B\lambda}{2} + \frac{B\lambda}{6}\right]}\right)=\exp\left(-\dfrac{3\lambda}{16B}\right)
	\end{align*}
	
\end{proof}

\section{The First Scheme: Model}
The proposed scheme for allocating $n$ balls into $n$ bins is as follows: Consider a connected $d$-regular graph $G=(V, E)$ with $|V| = n$. We associate bins with distinct vertices on the graph. Let $W_1(t), W_2(t)$ be processes on the graph which we will define later. At each time $t=0,1,\cdots, n-1$, a ball is inserted into the least loaded bin of $W_1(t)$ and $W_2(t)$, where ties are broken arbitrarily (but not randomly).

$W_1(t), W_2(t)$ are defined as follows: Starting from independent uniform random vertices, two random walkers perform non-backtracking random walks independently on the graph. At an time, if either two random walkers' path intersects\footnote{At least one walker arrives at a node which was visited in the past by one of the walkers.}, or a restart has not taken place for $\lfloor c\log n\rfloor$ timestamps, then a ``restart" happens, i.e. the two random walkers are reset to independent uniform vertices on the graph. The pathes of the two walkers are reset. The above process is repeated until time $n$. 

\subsection{Model Setup}
Let $\mathbf{Q}(t) = (Q_1(t), \cdots, Q_n(t))$ be the bin load vector at time $t$, where $t=0,1,\cdots$. We assume that the system is initially empty, i.e. $\mathbf{Q}(0) = \mathbf{0}$. For $0\leq t< n$, we have
\begin{equation}
\begin{split}
\mathbf{Q}(t+1) = \begin{cases}
\mathbf{Q}(t) + e_{W_1(t)}&\text{if }Q_{W_1(t)}(t)< Q_{W_2(t)}(t)\\
\mathbf{Q}(t) + e_{W_2(t)}&\text{if }Q_{W_1(t)}(t)> Q_{W_2(t)}(t)\\
\text{either one of the above expressions}&\text{if }Q_{W_1(t)}(t)= Q_{W_2(t)}(t)
\end{cases}
\end{split}
\end{equation}

We also define $\mathbf{Q}(t) = \mathbf{Q}(n)$ for all $t\geq n$.

Mathematically $W_1(t), W_2(t)$ can be defined as follows: For each vertex $i\in V$, we associate the $d$ neighbors of $i$ with a strict ordering. Let $N(i)$ denote the set of neighbors of $i$. For $l\in [k-1]$, define $\tilde{T}(i, j, l)$ to be the $l$-th ranked neighbor of $i$ in $N(i)\backslash\{j\}$. For each directed edge $e=(j, i)\in E$, we define $T(e, l) := (T_1(e, l), T_2(e, l)) := (i, \tilde{T}(i, j, l))$.

Let $D_1(t), D_2(t), t=1,2,\cdots$ be i.i.d. uniform random variables on $\{1,\cdots, k-1\}$. Let $R_1(t), R_2(t), t=0,1,2,\cdots$ be i.i.d. uniform random directed edges. $D_j(s), R_j(t), j=1,2,s=1,2,\cdots, t=0,1,2,\cdots$ are assumed to be mutually independent.

Let $\Psi_1(t), \Psi_2(t)\in E$, $\mathbf{M}(t)\in \{0, 1\}^n$, $C(t)\in\mathbb{Z}_+$. We now define $\Psi_j(0) = R_j(0), j=1,2$, $\mathbf{M}(0) = e_{R_{12}(0)} \vee e_{R_{22}(0)} $, $C(0) = 0$. For $t\geq 1$,

\begin{align}
J(t)&:=\begin{cases}
1&\text{if }M_{T_2(\Psi_j(t-1), D_j(t))}=1~\text{for some }j=1,2~\text{or~}C(t-1) = \lfloor c\log n \rfloor-1\\
0&\text{otherwise}
\end{cases}\\
\Psi_j(t) &= \begin{cases}
T(\Psi_j(t-1), D_j(t))&\text{if } J(t) = 0 \\
R_j(t)&\text{if }J(t) = 1
\end{cases}\\
\textbf{M}(t)&=\begin{cases}
\textbf{M}(t-1) \vee e_{\Psi_{12}(t) } \vee e_{\Psi_{22}(t) } &\text{if~} J(t) = 0\\
e_{R_{12}(t)} \vee e_{R_{22}(t)}&\text{if~}J(t)=1
\end{cases}\\
C(t)&=\begin{cases}
C(t-1) + 1&\text{if }J(t) = 0\\
0&\text{if }J(t) = 1
\end{cases}
\end{align}

Finally we define $W_j(t):=\Psi_{j2}(t)$ (i.e. the end vertex of the directed edge $\Psi_j(t)$) for $j=1,2,t=0,1,2,\cdots$.

Define the filtration 
\begin{equation}
\mathcal{F}_t:=\sigma(\{D_1(s), D_2(s) \}_{s=1}^t, \{R_1(s), R_1(s) \}_{s=0}^{t-1} ),~~~~~~t=0,1,\cdots
\end{equation} 

It is immediate that the random walker positions $\{W_1(s), W_2(s)\}_{s=0}^{t-1}$ are measurable with respect to $\mathcal{F}_t$. As a consequence, the queue length vector $\mathbf{Q}(t)$ is measurable with respect to $\mathcal{F}_t$.

\section{The First Scheme: Main Results}\label{sec:main-results}
\begin{thm}
	Under Condition \eqref{nllrst} specified in Assumption \ref{assump: 1} which we will specify later, the maximum load achieved by the proposed scheme is less than $\log_2\log n + \Theta(1)$ with high probability.
\end{thm}

The main idea of the proof is to construct the resets such that probability of intersections of the paths of the walkers is made sufficiently small, so that thereafter, the assignment of the balls is made just as in the regular \textit{power-of-d choices} scheme. Note that having the walkers reset to the stationary distribution of the non-backtracking random walk is critical to latter property. The proof uses the iterative bounding technique outlined in \cite{azar1999balanced}, while the dependent sampling enforced by the random walks forces the use of martingale concentration inequalities to bound the number of bins with high load.

\subsection{Preliminary Results}\label{sec:prelim}
Define $T_j$ to be the time of the $j$-th restart, i.e. the $j$-th smallest $t$ such that $J(t)=1$. Also define $T_0\equiv 0$. $T_j$ is a stopping time with respect to $\mathcal{F}_t$. Define a new filtration 
\begin{equation}
\mathcal{G}_j:=\mathcal{F}_{T_j}:=\{A\in \mathcal{F}_\infty:A\cap \{T_j\leq t \}\in\mathcal{F}_t~\text{for all}~ t=0,1,\cdots \},~~~~~~j=0,1,2,\cdots
\end{equation}

Let $\nu^{(i)}(t)$ denote the number of bins with load at least $i$ at time $t$. Define the height of a ball $j$ to be number of balls in the bin that ball $j$ is inserted into after insertion of ball $j$. Let $\mu^{(i)}(t)$ denote the number of balls within the first $t$ balls with height at least $i$. Clearly, we have $\mu^{(i)}(t)\geq \nu^{(i)}(t)$ for all $i$ and all $t$.

Define the events
\begin{equation}
\mathcal{E}_i:=\{ \nu^{(i)}(n) \leq \beta_i \},~~~~~~i=1,2,\cdots
\end{equation}
where $\beta_i\in \mathbb{R}_+$ will be specified later.

Define
\begin{equation}\label{defI}
I_{j}^{(i)}(s):=\begin{cases}
1&\text{if }Q_{W_l(T_j+s)}(T_j+s)\geq i~~\forall l=1,2 \\
&\text{ and }T_j+s<T_{j+1}\text{ and } \nu^{(i)}(T_j)\leq \beta_i\\
0&\text{otherwise}
\end{cases}
\end{equation}

Also define
\begin{equation}
Z_j^{(i)} = \sum_{s=0}^{\lfloor c\log n\rfloor - 1} I_j^{(i)}(s)
\end{equation}

Clearly, condition on $\mathcal{E}_i$, $\nu^{(i)}(T_j)\leq \beta_i$ is true, and we have $Z_j^{(i)}$ to be equal to the number of occasions where both sampled bins has load at least $i$ between the $j$-th and $(j+1)$-th restart of the random walk.

\begin{lemma}
	For all $s=0,1,\cdots, \lfloor c\log n\rfloor - 1$
	\begin{equation}
	\E[I_j^{(i)}(s)~|~\mathcal{G}_j] \leq \left(\dfrac{\beta_i}{n}\right)^2~~~~~~\text{a.s.}
	\end{equation}
\end{lemma}

\begin{proof}
	At the time of $j$-th restart, $\{\mathbf{Q}(t)\}_{t\geq T_j}, \{\mathbf{W}(t) \}_{t\geq T_j}$ are conditionally independent of $\mathcal{G}_j$ given $\mathbf{Q}(T_j)$. Hence
	\begin{equation}\label{k0}
	\E[I_j^{(i)}(s)~|~\mathcal{G}_j] = \E[I_j^{(i)}(s)~|~\mathbf{Q}(T_j)]
	\end{equation}
	
	Define $\widetilde{\mathbf{W}}(t)$ as follows: $\widetilde{\mathbf{W}}(t) = \mathbf{W}(t)$ for all $t\leq T_j$. After $T_j$, define the evolution of $\widetilde{W}(t)$ as follows: Let $\widetilde{\Psi}(T_j) = \Psi(T_j)$. Then
	\begin{align}
	\widetilde{\Psi}_j(t) &:= T(\widetilde{\Psi}_j(t-1), D_j(t))~~~~~~\forall t>T_j
	\end{align}
	then set $\widetilde{W}_j(t) = \Psi_{j2}(t)$. In other words, for $\widetilde{\mathbf{W}}(t)$, the restart mechanism is canceled after $T_j$. It is clear that $\widetilde{\mathbf{W}}(t) = \mathbf{W}(t)$ is also true for $T_j\leq t < T_{j+1}$.
	
	If $\mathbf{Q}(T_j)=\mathbf{q}$ is such that $\nu^{(i)}(T_j)>\beta_i$, then
	\begin{equation}\label{k1}
	\E[I_j^{(i)}(s)~|~\mathbf{Q}(T_j) = \mathbf{q}] = 0
	\end{equation}
	by definition in \eqref{defI}. 
	
	If $\mathbf{Q}(T_j)=\mathbf{q}$ is such that $\nu^{(i)}(T_j)\leq \beta_i$, then
	\begin{equation}
	\begin{split}
	\E[I_j^{(i)}(s)~|~\mathbf{Q}(T_j) = \mathbf{q}] &= \Pr(Q_{W_l(T_j+s)}(T_j+s)\geq i~\forall l=1,2,~T_j+s<T_{j+1}~|~\mathbf{Q}(T_j) = \mathbf{q})\\
	&\stackrel{(*)}{=} \Pr(Q_{W_l(T_j+s)}(T_j)\geq i~\forall l=1,2,~T_j+s<T_{j+1}~|~\mathbf{Q}(T_j) = \mathbf{q})\\
	&= \Pr(Q_{\widetilde{W}_l(T_j+s)}(T_j)\geq i~\forall l=1,2,~T_j+s<T_{j+1}~|~\mathbf{Q}(T_j) = \mathbf{q})\\
	&\leq\Pr(Q_{\widetilde{W}_l(T_j+s)}(T_j)\geq i~\forall l=1,2~|~\mathbf{Q}(T_j) = \mathbf{q})\\
	&=\Pr(q_{\widetilde{W}_l(T_j+s)}\geq i~\forall l=1,2)
	\end{split}
	\end{equation}
	where (*) is true since by construction, $W_1(T_j+s)$ and $W_2(T_j+s)$ are vertices that have never been visited in period $[T_j, T_j+s)$, given that $T_j+s<T_{j+1}$. Hence the load of these bins at time $T_j+s$ are the same as at time $T_j$.
	
	For any $s\geq 0$, $\widetilde{W}_1(T_j+s), \widetilde{W}_2(T_j+s)$ are uniformly and independently distributed on the set of vertices, hence
	\begin{equation}\label{k2}
	\begin{split}
	\E[I_j^{(i)}(s)~|~\mathbf{Q}(T_j) = \mathbf{q}]&\leq\Pr(q_{\widetilde{W}_l(T_j+s)}\geq i~\forall l=1,2)\\&=\Pr(q_{\widetilde{W}_1(T_j+s)}\geq i) \Pr(q_{\widetilde{W}_2(T_j+s)}\geq i)\\
	&\leq\left( \dfrac{\beta_i}{n}\right)^2 
	\end{split}
	\end{equation}
	
	Combining \eqref{k0}, \eqref{k1} and \eqref{k2}, we conclude that
	\begin{equation}
	\E[I_j^{(i)}(s)~|~\mathcal{G}_j] \leq \left( \dfrac{\beta_i}{n}\right)^2
	\end{equation}
\end{proof}

Denote $N = \lfloor \frac{en}{2c\log n}\rfloor$, then define the event
\begin{equation}
A:=\{ T_{N} > n \}
\end{equation}
i.e. the event that at most $N$ restarts (including the initial start) has taken place within time $n$.

Condition on $\mathcal{E}_i\cap A$, we have $\sum_{l=0}^{N-1} Z_l^{(i)}$ to be the number of occasions where both sampled bins has load at least $i$ before time $T_{N}$, which is larger than $n$. Thus $\sum_{l=0}^{N-1} Z_l^{(i)} \geq \mu^{(i+1)}(n)$ on $\mathcal{E}_i\cap A$.

To bound the probability of $A^c$, we need to impose some condition on the graph $G$ (more specifically, a sequence of graphs $G_n$), 

\begin{assump}\label{assump: 1}
	\begin{align}\label{nllrst}
	\Pr(T_1 < \lfloor c\log n\rfloor ) \leq 0.1 \tag{N}
	\end{align}
	i.e. intersection before time $\lfloor c\log n\rfloor$ of randomly initialized non-backtracking random walks on the graph is not of high probability.
\end{assump}

\begin{lemma}\label{lem:nreset}
	Under Condition \eqref{nllrst} specified in Assumption \ref{assump: 1}, we have $\Pr(A^c) \leq \exp\left(\frac{n}{340c\log n}\right)$ for large $n$.
\end{lemma}

\begin{remark}
	The choice of 0.1 in \eqref{nllrst} is purely arbitrary. With more careful analysis one can allow the bound to be any number strictly smaller than 1. However, for most interesting cases, the probability $\Pr(T_1<\lfloor c\log n\rfloor)$ goes to 0 as $n$ goes to infinity.
\end{remark}

\begin{remark}
	Assume that $c$ is a constant. If the graph $G$ has girth $g > 6 \log_{k-1} \log n$, then (\ref{nllrst}) is true for sufficiently large $n$. If $G$ is a $k$-regular random graph, then (\ref{nllrst}) is true with high probability.
\end{remark}

\subsection{Proof of Main Theorem}

Set $\beta_6 = \dfrac{n}{2e}$. Then $\nu^{(6)}(n)\leq \beta_i$ is always true. Hence $\Pr(\mathcal{E}_6) = 1$.

First for all $i$, we have
\begin{align*}
\Pr(\mathcal{E}_{i+1}^c \cap \mathcal{E}_i \cap A) &= \Pr(\nu^{(i+1)}(n) > \beta_{i+1}, \mathcal{E}_i\cap A )\\
&\leq \Pr(\mu^{(i+1)}(n) > \beta_{i+1}, \mathcal{E}_i\cap A )\\
&\leq \Pr\left(\sum_{j=0}^{N-1} Z_j^{(i)} > \beta_{i+1}, \mathcal{E}_i\cap A\right)\\
&\leq \Pr\left(\sum_{j=0}^{N-1} Z_j^{(i)} \geq\beta_{i+1}\right)
\end{align*}

We know that
\begin{equation}
\E[Z_j^{(i)}~|~\mathcal{G}_j]\leq c\log n\left(\dfrac{\beta_i}{n}\right)^2=:m_i
\end{equation}

and
\begin{equation}
\dfrac{en}{2} \geq Nc\log n
\end{equation}

For $6\leq i< i^*$, set $\beta_{i+1} = en\left(\frac{\beta_i}{n}\right)^2$, where $i^*$ is the first $i$ such that $en\left(\frac{\beta_i}{n}\right)^2< 8c(\log n)^2$. It can be shown that $i^* \leq \log_2\log n + \Theta(1)$

We have $\beta_{i+1}\geq 2Nm_i$, hence we can apply Corollary \ref{berncor} and obtain
\begin{equation}\label{i1}
\begin{split}
\Pr(\mathcal{E}_{i+1}^c \cap \mathcal{E}_i \cap A) 
&\leq \Pr\left(\sum_{j=0}^{N-1} Z_j^{(i)} \geq \beta_{i+1}\right)\leq \exp\left( -\dfrac{3\beta_{i+1}}{16c\log n}\right)\\
&\leq \exp\left( -\dfrac{3\cdot 8c(\log n)^2}{16c\log n}\right) = n^{-1.5}
\end{split}
\end{equation}
for $6\leq i<i^*$.

Set $\beta_{i^*+1} = 8c(\log n)^2$. We still have $\beta_{i^*+1}\geq 2Nm_{i^*}$, applying Corollary \ref{berncor} again we have
\begin{equation}\label{istar}
\begin{split}
\Pr(\mathcal{E}_{i^*+1}^c \cap \mathcal{E}_{i^*} \cap A) 
&\leq \Pr\left(\sum_{j=0}^{N-1} Z_j^{(i^*)} \geq \beta_{i^*+1}\right)\leq \exp\left( -\dfrac{3\beta_{i^*+1}}{16c\log n}\right)\\
&= \exp\left( -\dfrac{3\cdot 8c(\log n)^2}{16c\log n}\right) = n^{-1.5}
\end{split}
\end{equation}

Set $\beta_{i^*+2} = 0.8$, by Markov Inequality we have
\begin{equation}\label{istarp1}
\begin{split}
\Pr(\mathcal{E}_{i^*+2}^c \cap \mathcal{E}_{i^*+1} \cap A) 
&\leq \Pr\left(\sum_{j=0}^{N-1} Z_j^{(i^*+1)} \geq \beta_{i^*+2}\right)\leq \dfrac{\E\left[\sum_{j=0}^{N-1} Z_j^{(i^*+1)} \right] }{\beta_{i^*+2}}\\&\leq \dfrac{Nc\log n \left(\frac{\beta_{i^*+1}}{n}\right)^2 }{\beta_{i^*+2}}\leq\dfrac{\frac{en}{2} \left(\frac{\beta_{i^*+1}}{n}\right)^2 }{\beta_{i^*+2}}=40e\cdot\dfrac{c^2(\log n)^2}{n}
\end{split}
\end{equation}

Hence, combining \eqref{i1}\eqref{istar}\eqref{istarp1} we obtain
\begin{equation*}
\begin{split}
\Pr(\mathcal{E}_{i^*+2}^c ) &= \Pr(\mathcal{E}_{i^*+2}^c\cap \mathcal{E}_6)\\
&\leq \Pr(\mathcal{E}_{i^*+2}^c\cap \mathcal{E}_6\cap A ) + \Pr(A^c) \\
&\leq \Pr\left( \bigcup_{i=6}^{i^*+1} (\mathcal{E}_{i+1}^c \cap \mathcal{E}_i\cap A)  \right)+ \Pr(A^c)\\
&\leq \sum_{i=6}^{i^*+1}\Pr(\mathcal{E}_{i+1}^c \cap \mathcal{E}_i\cap A) + \Pr(A^c)\\
&\leq \dfrac{i^*-4}{n^{1.5}} + 40e\cdot \dfrac{c^2(\log n)^4}{n} + \exp\left(-\dfrac{n}{340c\log n}\right)\\
&=o(1)
\end{split}
\end{equation*}
for $n\geq \max\{\exp(\frac{10}{c}), \left(\frac{40c}{\alpha e}\right)^{1/(1-\alpha)},3\}$ for any $0<\alpha<1$.

Therefore, with high probability, $\mathcal{E}_{i^*+2}$ is true, that is, no bin has load exceeding $i^*+2\leq \log_2\log n + \Theta(1)$. Proving the result.

\begin{remark}
	All the above results are non-asymptotic. One can consider some sequence $c_n$ that grows or decays as $n\rightarrow\infty$, as long as $n\geq \max\{\exp(\frac{10}{c}), \left(\frac{40c}{\alpha e}\right)^{1/(1-\alpha)}\}$ is asymptotically true for some $\alpha\in (0, 1)$. In fact, if $c_n$ is such that $c_n \log n$ are integer valued, then one only needs to ensure that $n\geq  \max\{ \exp(\frac{1}{c}),\left(\frac{40c}{\alpha e}\right)^{1/(1-\alpha)} \}$ for some $\alpha\in (0, 1)$
\end{remark}

\begin{remark}
	By setting $c=\dfrac{1}{\log n}$, The scheme becomes the same as the classical power of 2 scheme. We hence obtain the classical result that the maximum load under power of 2 is $\log_2\log n + \Theta(1)$ with high probability. 
\end{remark}

\begin{cor}[Round-Robin with Restart]
	Let $G$ be a cycle graph and $c\log n = \lfloor n^{0.4}\rfloor$. Then the maximum load under the proposed scheme is less than $\log_2\log n + \Theta(1)$ with high probability.
\end{cor}

\section{The Second Scheme: Model}\label{sec:simp}
In this section we introduce a simpler scheme: In this scheme we still have a underlying $k$-regular graph $G$ and a process $(W_1(t), W_2(t))$ defined on the graph. At each time, a ball is assigned to the least loaded bin of $W_1(t)$ and $W_2(t)$. 

$W_1(t), W_2(t)$ are now defined as follows: Starting from independent uniform positions, two random walkers perform independent non-backtracking simple random walks on graph $G$. In contrast to the first scheme, we allow intersections: A restart only happens every $\lfloor c\log n\rfloor$ periods, i.e. every $\lfloor c\log n\rfloor$ timestamps, the two random walkers are reset to independent uniform vertices on the graph. We will assume that $c$ is a positive constant for this scheme.

For this scheme, we need a stronger assumption on the graph $G$: 

\begin{assump}\label{assump: 2}
	The girth of the graph is greater than or equal to $2\lceil \alpha\log_{k-1} n \rceil + 1$, where $\alpha$ is a strictly positive constant.
\end{assump}

\section{The Second Scheme: Main Results}
\begin{thm}\label{thm:2works}
	Under Assumption \ref{assump: 2}, the maximum load achieved by the scheme proposed in Section \ref{sec:simp} is less than $\log_2\log n + \Theta(1)$ with high probability.
\end{thm}

Similar to the previous scheme, the proof of Theorem \ref{thm:2works} is based on a key lemma which estimates the probability of sampling high load bins. The rest of the proof is based on martingale concentration inequalities and Azar et al.'s iterative bounding technique \cite{azar1999balanced}.

\subsection{Preliminary Results}
Unless specified, we will use the same notations as for the proof of the first scheme.

Define $T_j$ to be the time of the $j$-th restart. Now, we have $T_j=j\cdot \lfloor c\log n\rfloor$ deterministically. Also define $T_0\equiv 0$. Define a new filtration 
\begin{equation}
\mathcal{G}_j:=\mathcal{F}_{T_j},~~~~~~j=0,1,2,\cdots
\end{equation}

Let $\nu^{(i)}(t)$ denote the number of bins with load at least $i$ at time $t$. Define the height of a ball $j$ to be number of balls in the bin that ball $j$ is inserted into after insertion of ball $j$. Let $\mu^{(i)}(t)$ denote the number of balls with height at least $i$ at time $t$. Clearly, we have $\mu^{(i)}(t)\geq \nu^{(i)}(t)$ for all $i$ and all $t$.

Define the events
\begin{equation}
\mathcal{E}_i:=\{ \nu^{(r_i)}(n) \leq \beta_i \},~~~~~~i=1,2,\cdots
\end{equation}
where 
\begin{equation}
r_i=\begin{cases}
i&i\leq i^*\\
i^*+(i-i^*)(L+1)&i>i^*
\end{cases}
\end{equation}
where $i^*\in\mathbb{N}, L\in\mathbb{Z}_+$ and $\beta_i\in \mathbb{R}_+$ will be specified later.

Define
\begin{equation}\label{defI2}
I_{j}^{(i)}(s):=\begin{cases}
1&\text{if }Q_{W_l(T_j+s)}(T_j+s)\geq r_{i+1}-1~~\forall l=1,2 \\
&\text{ and } \nu^{(r_{i})}(T_j) \leq \beta_{i}\\
0&\text{otherwise}
\end{cases}~~~~~~0\leq s <\lfloor c\log n\rfloor
\end{equation}
for $i\geq 2$.

Also define
\begin{equation}
Z_j^{(i)} = \sum_{s=0}^{\lfloor c\log n\rfloor - 1} I_j^{(i)}(s)
\end{equation}

Clearly, condition on $\mathcal{E}_i$, $\nu^{(i)}(T_j)\leq \beta_i$ is true, and we have $Z_j^{(i)}$ to be equal to the number of occasions where both sampled bins has load at least $r_{i+1}-1$ between the $j$-th and $(j+1)$-th restart of the random walk, which equals the increment of number of balls with height at least $r_{i+1}$ in this period.

\begin{lemma}\label{lem:alon}
	Let $V(t)$ be a non-backtracking random walk on $G$, starting from a uniform random vertex. Then for all $s>t$
	\begin{equation}
	\Pr(V(s) = V(t)~|~V(t), V(t-1)) \leq \dfrac{1}{n^{\alpha}}\qquad a.s.
	\end{equation}
\end{lemma}

\begin{proof}
	This proof is the same as \cite{alon2007non}. We rewrite the proof for our setting. First, if $s-t < 2\lceil \alpha \log_{k-1} n \rceil$, then $\Pr(V(s) = V(t)~|~V(t), V(t-1)) = 0$.
	
	Otherwise $s-t \geq 2\lceil \alpha \log_{k-1} n \rceil$. Let $v$ be an arbitrary vertex. Notice that the neighborhood of $v$ up to distance $h=\lceil \alpha \log_{k-1} n \rceil$ is a $k$-regular tree. Let $U_v$ denote the set of $k(k-1)^{h-1}$ leaves of the tree. Since the walk on the graph is non-backtracking, $V(s)=v$ implies that $V(s - \lceil \alpha \log_{k-1} n \rceil) \in U_v$. Therefore
	\begin{equation}
	\begin{split}
	\Pr(V(s) = v~|~V(t) = v, V(t-1)) &= \Pr(V(s) = v, V(s - \lceil \alpha \log_{k-1} n \rceil) \in U_v ~|~V(t) = v, V(t-1))\\
	&\leq \Pr(V(s) = v~|~V(s - \lceil \alpha \log_{k-1} n \rceil) \in U_v, V(t) = v, V(t-1))\\
	&\leq \left(\dfrac{1}{k-1}\right)^{\lceil \alpha \log_{k-1} n \rceil}\leq \dfrac{1}{n^{\alpha}}
	\end{split}
	\end{equation}
	and lemma follows since the above is true for all vertex $v$.
	
\end{proof}

\begin{lemma}\label{lem:key2}
	For all $s=0,1,\cdots, \lfloor c\log n\rfloor - 1$
	\begin{equation}\label{eq66}
	\E[I_j^{(i)}(s)~|~\mathcal{G}_j] \leq \left(\dfrac{\beta_i}{n}\right)^2 + \dfrac{2s}{n^{\alpha}} + \dfrac{2s}{n} ~~~~~~\text{a.s.}
	\end{equation}
	for all $i< i^*$.
	
	Furthermore, let $L:=\left\lfloor 2\left(\frac{c\log(k-1)}{2\alpha} + 1\right) \right\rfloor$, we have
	\begin{equation}\label{eq67}
	\E[I_j^{(i)}(s)~|~\mathcal{G}_j] \leq \left(\dfrac{\beta_{i}}{n}\right)^2~~~~~~\text{a.s.}
	\end{equation}
	for all $i\geq i^*$.
\end{lemma}

\begin{proof}
	At the time of $j$-th restart, $\{\mathbf{Q}(t)\}_{t\geq T_j}, \{\mathbf{W}(t) \}_{t\geq T_j}$ are conditionally independent of $\mathcal{G}_j$ given $\mathbf{Q}(T_j)$. Hence
	\begin{equation}\label{k02}
	\E[I_j^{(i)}(s)~|~\mathcal{G}_j] = \E[I_j^{(i)}(s)~|~\mathbf{Q}(T_j)]
	\end{equation}
	holds for all $i$.
	
	Define $\mathcal{D}_{j, s}$ to be the event that at least one of $W_1(T_j+s)$ and $W_2(T_j+s)$ was visited at least once between time $T_j$ and $T_j+s$. By Union Bound we have
	\begin{equation}
	\begin{split}
	\Pr(\mathcal{D}_{j, s}) &\leq \sum_{r=0}^{s-1}\Pr(W_1(T_j+r) = W_1(T_j+s)) + \sum_{r=0}^{s-1}\Pr(W_1(T_j+r) = W_2(T_j+s)) \\&~~~~+ \sum_{r=0}^{s-1}\Pr(W_2(T_j+r) = W_1(T_j+s)) + \sum_{r=0}^{s-1}\Pr(W_2(T_j+r) = W_2(T_j+s))\\
	&=2\sum_{r=0}^{s-1}\Pr(W_1(T_j+r) = W_1(T_j+s)) + 2\sum_{r=0}^{s-1}\Pr(W_1(T_j+r) = W_2(T_j+s))\\
	&\leq \dfrac{2s}{n^{\alpha}} + \dfrac{2s}{n}
	\end{split}
	\end{equation}
	It is clear that $\mathcal{D}_{j, s}$ is not dependent on the queue lengths at time $T_j$.
	
	For all $i$, if $\mathbf{Q}(T_j)=\mathbf{q}$ is such that $\nu^{(r_i)}(T_j)>\beta_{i}$, then
	\begin{equation}\label{k12}
	\E[I_j^{(i)}(s)~|~\mathbf{Q}(T_j) = \mathbf{q}] = 0
	\end{equation}
	by definition in \eqref{defI}. 
	
	For $i<i^*$, if $\mathbf{Q}(T_j)=\mathbf{q}$ is such that $\nu^{(r_i)}(T_j)\leq \beta_{i}$, then
	\begin{equation}\label{k22}
	\begin{split}
	\E[I_j^{(i)}(s)~|~\mathbf{Q}(T_j) = \mathbf{q}] &= \Pr(Q_{W_l(T_j+s)}(T_j+s)\geq r_{i+1} - 1~\forall l=1,2|~\mathbf{Q}(T_j) = \mathbf{q})\\
	&= \Pr(Q_{W_l(T_j+s)}(T_j+s)\geq i~\forall l=1,2|~\mathbf{Q}(T_j) = \mathbf{q})\\
	&= \Pr(Q_{W_l(T_j+s)}(T_j+s)\geq i~\forall l=1,2,~\mathcal{D}_{j, s}^c|~\mathbf{Q}(T_j) = \mathbf{q}) \\
	&~~~~+ \Pr(Q_{W_l(T_j+s)}(T_j+s)\geq i~\forall l=1,2,~\mathcal{D}_{j, s}|~\mathbf{Q}(T_j) = \mathbf{q})\\
	&= \Pr(q_{W_l(T_j+s)}\geq i~\forall l=1,2,~\mathcal{D}_{j, s}^c|~\mathbf{Q}(T_j) = \mathbf{q}) \\
	&~~~~+ \Pr(Q_{W_l(T_j+s)}(T_j+s)\geq i~\forall l=1,2,~\mathcal{D}_{j, s}|~\mathbf{Q}(T_j) = \mathbf{q})\\
	&\leq \Pr(q_{W_l(T_j+s)}\geq i~\forall l=1,2|~\mathbf{Q}(T_j) = \mathbf{q}) + \Pr(\mathcal{D}_{j, s}|~\mathbf{Q}(T_j) = \mathbf{q})\\
	&= \Pr(q_{W_l(T_j+s)}\geq i~\forall l=1,2) + \Pr(\mathcal{D}_{j, s})\\
	&\leq \left( \dfrac{\beta_i}{n}\right)^2  + \dfrac{2s}{n^{\alpha}} + \dfrac{2s}{n}
	\end{split}
	\end{equation}
	
	Combining \eqref{k02}\eqref{k12}\eqref{k22}, we conclude that
	\begin{equation}
	\E[I_j^{(i)}(s)~|~\mathcal{G}_j] \leq \left( \dfrac{\beta_i}{n}\right)^2 + \dfrac{2s}{n^{\alpha}} + \dfrac{2s}{n}
	\end{equation}
	for $i<i^*$
	
	To prove the second statement, we notice that it is impossible for a random walker to visit one vertex strictly more than $\dfrac{\lfloor c\log n\rfloor}{2\lceil\alpha \log_{k-1}n\rceil +1 } + 1$ times within $\lfloor c\log n\rfloor$ timestamps. Hence for any $0\leq s < \lfloor c\log n\rfloor$, both $W_1(T_j+s)$ and $W_2(T_j+s)$ has been visited no more than
	\begin{equation}
	2\cdot \left(\dfrac{\lfloor c\log n\rfloor}{2\lceil\alpha \log_{k-1}n\rceil +1 } + 1\right)<2\left(\dfrac{c \log(k-1)}{\alpha} + 1 \right)
	\end{equation}
	times.
	
	Therefore, $Q_{W_l(T_j+s)}(T_j+s)\leq q_{W_l(T_j+s)} + L$ almost surely for $l=1,2$. 
	
	For $i\geq i^*$, if $\mathbf{Q}(T_j)=\mathbf{q}$ is such that $\nu^{(r_i)}(T_j)\leq \beta_{i}$, then
	\begin{equation}\label{k32}
	\begin{split}
	\E[I_j^{(i)}(s)~|~\mathbf{Q}(T_j) = \mathbf{q}] &= \Pr(Q_{W_l(T_j+s)}(T_j+s)\geq r_{i+1} - 1~\forall l=1,2~|~\mathbf{Q}(T_j) = \mathbf{q})\\
	&= \Pr(Q_{W_l(T_j+s)}(T_j+s)\geq r_{i} + L~\forall l=1,2~|~\mathbf{Q}(T_j) = \mathbf{q})\\
	&\leq \Pr(q_{W_l(T_j+s)} + L\geq r_i+L~\forall l=1,2~|~\mathbf{Q}(T_j) = \mathbf{q})\\
	&= \Pr(q_{W_l(T_j+s)} \geq r_i~\forall l=1,2)\\
	&= \Pr(q_{W_1(T_j+s)} \geq r_i) \Pr(q_{W_2(T_j+s)} \geq r_i)\\
	&\leq \left(\dfrac{\beta_{i}}{n}\right)^2
	\end{split}
	\end{equation}
	
	Combining \eqref{k02}\eqref{k12}\eqref{k32}, we conclude that
	\begin{equation}
	\E[I_j^{(i)}(s)~|~\mathcal{G}_j] \leq \left(\dfrac{\beta_{i}}{n}\right)^2
	\end{equation}
	for $i\geq i^*$.
	
\end{proof}

\subsection{Proof of the Main Theorem}
Different from the proof of the previous scheme, in this scheme we have a deterministic inter-reset time. Hence, reset happens exactly $N = \lceil \frac{n}{\lfloor c\log n\rfloor }\rceil$ many times. For large $n$, we have
\begin{equation}
N\leq \dfrac{en}{2c\log n}
\end{equation}

Conditioned on $\mathcal{E}_i$, we have $\sum_{l=0}^{N-1} Z_l^{(i)}$ to be the number of occasions where both sampled bins have load at least $r_{i+1} - 1$ before time $T_{N}\geq n$. Thus $\sum_{l=0}^{N-1} Z_l^{(i)} \geq \mu^{(r_{i+1})}(n)$ holds on $\mathcal{E}_i$.

Let $i^*> 9$ be determined later. Set $\beta_{9} = \dfrac{n}{3e}$. Then $\nu^{(r_9)}(n)\leq \beta_{9}$ is always true (as there cannot be more than $\frac{n}{3e}$ bins with load at least $9$). Hence $\Pr(\mathcal{E}_{9}) = 1$.

For $i\geq 9$, we have
\begin{align}
\Pr(\mathcal{E}_{i+1}^c \cap \mathcal{E}_i) &= \Pr(\nu^{(r_{i+1})}(n) > \beta_{i+1}, \mathcal{E}_i )\\
&\leq \Pr(\mu^{(r_{i+1})}(n) > \beta_{i+1}, \mathcal{E}_i )\\
&\leq \Pr\left(\sum_{j=0}^{N-1} Z_j^{(i)} > \beta_{i+1}, \mathcal{E}_i\right)\\
&\leq \Pr\left(\sum_{j=0}^{N-1} Z_j^{(i)} \geq \beta_{i+1}\right)
\end{align}

Now, for $9\leq i < i^*$, define $\beta_{i+1} = 2en\left(\frac{\beta_i}{n}\right)^2$, where $i^*$ is defined to be the smallest $i$ such that
\begin{equation}
\left(\frac{\beta_{i-1}}{n}\right)^2< \dfrac{2c\log n}{n^{\gamma}}
\end{equation}
where $\gamma:=\frac{1}{2}\wedge \alpha$. It can be shown that $i^*\leq \log_2\log n + \Theta(1)$.

For $9\leq i<i^*-1$ we have
\begin{equation}
\left(\frac{\beta_i}{n}\right)^2\geq \dfrac{2c\log n}{n^{1\wedge\alpha}}\geq \dfrac{c\log n}{n^\alpha} + \dfrac{c\log n}{n}
\end{equation}

Thus for $9\leq i<i^*-1$,
\begin{equation}
\begin{split}
\E[Z_j^{(i)}~|~\mathcal{G}_j] &\leq \sum_{s=0}^{\lfloor c\log n\rfloor - 1}\left[\left(\frac{\beta_i}{n}\right)^2 + \dfrac{2s}{n^\alpha} + \dfrac{2s}{n} \right]\\
&\leq c\log n \left(\frac{\beta_i}{n}\right)^2 + \dfrac{(c\log n)^2}{n^\alpha} + \dfrac{(c\log n)^2}{n}\\
&\leq 2c\log n \left(\frac{\beta_i}{n}\right)^2=:m_i
\end{split}
\end{equation}

For $9\leq i<i^*-1$ We check that $\beta_{i+1}=2\cdot\frac{en}{2c\log n}\cdot c\log n\left(\frac{\beta_i}{n}\right)^2 \geq 2Nm_i$. Hence applying Corollary \ref{berncor} we obtain
\begin{equation}
\begin{split}
\Pr(\mathcal{E}_{i+1}^c \cap \mathcal{E}_i)&\leq \Pr\left(\sum_{j=0}^{N-1} Z_j^{(i)} \geq \beta_{i+1} \right)\leq \exp\left(-\dfrac{3\beta_{i+1}}{16c\log n}\right)\\
&\leq \exp\left(-\dfrac{3\cdot 4ecn^{1-\gamma}\log n}{16c\log n}\right) = \exp(-\dfrac{12e}{16}n^{1-\gamma})\\
&\leq \exp(-n^{1/2})
\end{split}\label{www2}
\end{equation}
for $9\leq i< i^*-1$. 

Now we have
\begin{equation}
\begin{split}
\E[Z_j^{(i^*-1)}~|~\mathcal{G}_j] &\leq \sum_{s=0}^{\lfloor c\log n\rfloor - 1}\left[\left(\frac{\beta_{i^*-1}}{n}\right)^2 + \dfrac{2s}{n^\alpha} + \dfrac{2s}{n} \right]\\
&\leq c\log n \left(\frac{\beta_{i^*-1}}{n}\right)^2 + \dfrac{(c\log n)^2}{n^\alpha} + \dfrac{(c\log n)^2}{n}\\
&\leq \dfrac{2(c\log n)^2}{n^{\gamma}} + \dfrac{(c\log n)^2}{n^\alpha} + \dfrac{(c\log n)^2}{n}\\
&\leq \dfrac{4(c\log n)^2}{n^{\gamma}}=:m_{i^*-1}
\end{split}
\end{equation}

Set $\beta_{i^*} =4ecn^{1-\gamma}\log n$. We have $\beta_{i^*}= 2\cdot \frac{en}{2c\log n}\cdot  \frac{4(c\log n)^2}{n^{\gamma}}\geq 2Nm_{i^*-1}$. By applying Corollary \ref{berncor} again we have
\begin{align}
&~~~~\Pr(\mathcal{E}_{i^*}^c\cap \mathcal{E}_{i^*-1})\leq\Pr\left(\sum_{j=0}^{N-1} Z_j^{(i^*-1)} \geq \beta_{i^*}\right) \\
&=\exp\left(-\dfrac{3\beta_{i^*}}{16c\log n}\right)=\exp\left(-\dfrac{3\cdot 4ecn^{1-\gamma}\log n}{16c\log n}\right) = \exp(-\dfrac{12e}{16}n^{1-\gamma})\\
&\leq \exp(-n^{1/2})
\label{xxx2}
\end{align}

Now for $i^*\leq i < i^{**}-1$, set $\beta_{i+1} = en\left(\frac{\beta_{i}}{n}\right)^2$, where $i^{**}$ is the smallest $i\geq i^*+1$ such that $en\left(\frac{\beta_{i-1}}{n}\right)^2\leq 8c(\log n)^2$. It can be shown that $i^{**} - i^*\leq \Theta(1)$.

Now, we use \eqref{eq67} to bound the conditional expectation for $i\geq i^*$:
\begin{equation}
\E[Z_j^{(i)}~|~\mathcal{G}_j] \leq c\log n\left(\dfrac{\beta_{i}}{n}\right)^2=:m_i
\end{equation}

For $i^*\leq i < i^{**}-1$ we check that 
\begin{equation}
\beta_{i+1} = 2\cdot \dfrac{en}{2}\cdot c\log n \left(\dfrac{\beta_{i}}{n}\right)^2 \geq 2Nm_{i}
\end{equation}

Hence for $i^*\leq i < i^{**}-1$ we can apply Corollary \ref{berncor} and obtain
\begin{align}
&~~~~\Pr(\mathcal{E}_{i+1}^c \cap \mathcal{E}_{i}) \\
&\leq  \Pr\left(\sum_{j=0}^{N-1} Z_j^{(i)} \geq \beta_{i+1}\right)\\
&\leq \exp\left(-\dfrac{3\beta_{i+1}}{16c\log n} \right) \leq \exp\left(-\dfrac{3\cdot 8c(\log n)^2}{16c\log n} \right)\\
&\leq n^{-1.5}
\label{yyy2}
\end{align}

Set $\beta_{i^{**}} = 8c(\log n)^2$. We still have $\beta_{i^{**}}\geq en\left(\frac{\beta_{i^{**}-1}}{n}\right)^2\geq 2Nm_{i^{**}-1}$. Applying Corollary \ref{berncor} again we have
\begin{equation}\label{zzz2}
\begin{split}
&~~~~\Pr(\mathcal{E}_{i^{**}}^c \cap \mathcal{E}_{i^{**}-1})\\
&\leq \Pr\left(\sum_{j=0}^{N-1} Z_j^{(i^{**}-1)} \geq \beta_{i^{**}}\right)\leq \exp\left( -\dfrac{3\beta_{i^{**}}}{16c\log n}\right)\\
&= \exp\left( -\dfrac{3\cdot 8c(\log n)^2}{16c\log n}\right) = n^{-1.5}
\end{split}
\end{equation}

Set $\beta_{i^{**}+1} = 0.8$, by Markov Inequality we have
\begin{equation}\label{istarp22}
\begin{split}
&~~~~\Pr(\mathcal{E}_{i^{**}+1}^c \cap \mathcal{E}_{i^{**}})\\
&\leq \Pr\left(\sum_{j=0}^{N-1} Z_j^{(i^{**})} \geq \beta_{i^{**}+1}\right)\leq \dfrac{\E\left[\sum_{j=0}^{N-1} Z_j^{(i^{**})} \right] }{\beta_{i^{**}+1}}\\&\leq \dfrac{Nc\log n \left(\frac{\beta_{i^{**}}}{n}\right)^2 }{\beta_{i^{**}+1}}\leq\dfrac{\frac{en}{2} \left(\frac{\beta_{i^{**}}}{n}\right)^2 }{\beta_{i^{**}+1}}=40e\cdot\dfrac{c^2(\log n)^2}{n}
\end{split}
\end{equation}

Hence, combining \eqref{www2}\eqref{xxx2}\eqref{yyy2}\eqref{zzz2}\eqref{istarp22} we obtain
\begin{align}
\Pr(\mathcal{E}_{i^{**}+1}^c ) &= \Pr(\mathcal{E}_{i^{**}+1}^c\cap \mathcal{E}_9)\leq \Pr\left(\bigcup_{i=9}^{i^{**}}(\mathcal{E}_{i+1}^c\cap \mathcal{E}_i) \right)\leq \sum_{i=9}^{i^{**}}\Pr(\mathcal{E}_{i+1}^c\cap \mathcal{E}_i)\\
&\leq (i^*-9)\exp(-n^{1/2}) + (i^{**} - i^*)n^{-1.5} + 40e\cdot\dfrac{c^2(\log n)^2}{n}\\
&=o(1)
\end{align}

Therefore, with high probability, $\mathcal{E}_{i^{**}+1}$ is true, that is, no bin has load exceeding $$r_{i^{**} + 1} = i^* + (i^{**} - i^* + 1)(L+1) \leq \log_2 \log n + \Theta(1)$$proving the result.

\section{The Third Scheme: Model}
It turns out that we can design a scheme that is even simpler: When some stronger assumptions of the underlying graph are met, we do not need the random walks to reset at all! Specifically, the third scheme is as follows: Place the $t$-th ball into the least loaded bin of $W_1(t)$ and $W_2(t)$, where $W_1$ and $W_2$ are independent non-backtracking random walks on graph $G$, starting from independently uniform random vertices.

The idea comes from the fact that, under certain conditions for a graph, a non-backtracking random walk can mix within time $c\log n$, i.e. starting from any initial distribution on the graph, after time $c\log n$, the distribution of the random walker's location is close to uniform random. In this aspect, the third scheme is closely related to the second scheme.

\begin{defn}[Expander Graph]\label{def:expander}
	\cite{alon2007non} Let $\{G^{(n)}\}_{n\in I}$ be a sequence of $k$-regular graphs with $n$ vertices. Let $k=\lambda_1^{(n)}\geq \lambda_2^{(n)}\geq \cdots\geq \lambda_n^{(n)}$ be the eigenvalues of the adjacency matrix of $G^{(n)}$. Define $\lambda^{(n)} = \max\{\lambda_2^{(n)}, |\lambda_n^{(n)}| \}$. $\{G^{(n)}\}$ is called an expander graph sequence if the second largest eigenvalues $\lambda^{(n)}$ of the adjacency matrices of $G^{(n)}$ satisfies
	\begin{equation*}
	\lambda:=\limsup_{n\rightarrow\infty} \lambda^{(n)} < k
	\end{equation*}
	
\end{defn}

\begin{assump}[$G$ is a High Girth Expander]\label{assump: 3}
	The graph sequence $\{G^{(n)}\}$ is a $k$-regular expander graph sequence with girth of $G^{(n)}$ greater than $2\lceil \alpha\log_{k-1}n\rceil + 1$, where $\alpha$ is a positive constant.
\end{assump}

Such graphs do exist.

\begin{example}[LPS Graph]
	Lubotzky et al. \cite{lubotzky1988ramanujan} constructed a sequence of $(p+1)$-regular expander graphs with $\lambda \leq 2\sqrt{p}$ and girth greater than $\frac{4}{3}\log_{p}n$ asymptotically, where $p\geq 3$ is a prime number.
\end{example}

\begin{example}
	For even number $k\geq 4$, Gamburd et al. \cite{gamburd2009girth} show that the $k$-regular random Cayley graph of $\mathrm{SL}_2(\mathbb{F}_p)$ asymptotically almost surely has girth at least $(\frac{1}{3}-o(1))\log_{k-1}(|G|)$ as $p\rightarrow\infty$. In a separate paper by Bourgain and Gamburd \cite{bourgain2008uniform}, they show that $k$-regular random Cayley graph of $\mathrm{SL}_2(\mathbb{F}_p)$ are a.a.s. expanders.
\end{example}

\section{The Third Scheme: Main Results}\label{sec: 3}
\begin{thm}\label{thm:3works}
	Under Assumption \ref{assump: 3}, the maximum load achieved by the third scheme is less than $\log_2\log n + \Theta(1)$ with high probability.
\end{thm}

The proof idea for the third scheme is similar to the one for the second scheme. The basic framework is Azar et al.'s iterative bounding technique outlined in \cite{azar1999balanced}. For the concentration results, we divide the time into equally spaced ``mixing periods". When we look at only the even (resp. odd) periods, the mixing effect enables us to define a martingale. We bound the number of balls with height $i+1$ respectively for even and odd periods with martingale concentration inequality, and we finish up with a union bound. 

\subsection{Preliminary Results}
Define $\mathcal{F}_t = \sigma(\{W_1(s), W_2(s) \}_{s=0}^{t} )$, i.e. the smallest $\sigma$-algebra generated by the random walker paths before time $t$. As a result, $\mathbf{Q}(t)$ is measurable with respect to $\mathcal{F}_t$.

We first provide a mixing time result.
\begin{lemma}\label{lem:mixing}
	Let $V^{(n)}(t)$ be a non-backtracking random walk on expander graph $G^{(n)}$, then there exist constant $c>0$ (which only depends on $\lambda$ and $k$) such that
	\begin{equation}\label{mixingeq}
	\max_{u_0, u_1,v\in G^{(n)}}\Pr(V^{(n)}(t+1) = v~|~V^{(n)}(0) = u_0, V^{(n)}(1) = u_1 ) \leq \dfrac{2}{n}~~~~~~~~~\forall t\geq \T
	\end{equation}
	holds for all large $n$.
\end{lemma}

For the reset of the proof, Let $c>0$ be a constant such that \eqref{mixingeq} is true.

In this scheme we do not reset random walkers. However, we still define $T_j:=j\lfloor c\log n\rfloor$ for $j=0,1,\cdots$, except now $T_j$ is just some ``checkpoints" rather than times that some event happens. Accordingly, we define $\mathcal{G}_j:=\mathcal{F}_{T_j}$ for $j\in\mathbb{N}$. Define $\mathcal{G}_{-1} = \{\varnothing, \Omega \}$, i.e. the trivial $\sigma$-algebra. 

Let $\nu^{(i)}(t)$ denote the number of bins with load at least $i$ at time $t$, i.e. 
\begin{equation}
\nu^{(i)}(t) = \sum_{l=1}^n \mathbbm{1}_{\{Q_l(t) \geq i\}}
\end{equation}

Define the height of a ball $j$ to be number of balls in the bin that ball $j$ is inserted into after insertion of ball $j$. Let $\mu^{(i)}(t)$ denote the number of balls among the first $t$ balls with height at least $i$. Clearly, we have $\mu^{(i)}(t)\geq \nu^{(i)}(t)$ for all $i$ and all $t$.

Define the events
\begin{equation}
\mathcal{E}_i:=\{ \nu^{(r_i)}(n) \leq \beta_i \},~~~~~~i=1,2,\cdots
\end{equation}
where 
\begin{equation}
r_i=\begin{cases}
i&i\leq i^*\\
i^*+(i-i^*)(L+1)&i>i^*
\end{cases}
\end{equation}
where $i^*\in\mathbb{N}, L\in\mathbb{Z}_+$ and $\beta_i\in \mathbb{R}_+$ will be specified later.

Define
\begin{equation}\label{defI3}
I_{j}^{(i)}(s):=\begin{cases}
1&\text{if }Q_{W_l(T_j+s)}(T_j+s)\geq r_{i+1}-1~~\forall l=1,2 \\
&\text{ and } \nu^{(r_{i})}(T_j) \leq \beta_{i}\\
0&\text{otherwise}
\end{cases}~~~~~~0\leq s <\lfloor c\log n\rfloor
\end{equation}
for $i\geq 2$. $I_{j}^{(i)}(s)$ is measurable with respect to $\mathcal{G}_{j+1}$.

Also define
\begin{equation}
Z_j^{(i)} = \sum_{s=0}^{\lfloor c\log n\rfloor - 1} I_j^{(i)}(s)
\end{equation}

Clearly, condition on the event $\mathcal{E}_i$, $\nu^{(r_i)}(T_j)\leq \beta_i$ is true. Hence on event $\mathcal{E}_i$, $Z_j^{(i)}$ is the number of occasions where both sampled bins has load at least $r_{i+1}-1$ within time $[T_j, T_{j+1})$. By the scheme, a new ball is of height at least $r_{i+1}$ if and only if both sampled bins has load at least $r_{i+1}-1$, we conclude that on $\mathcal{E}_i$, $Z_j^{(i)}$ equals the increment of number of balls with height at least $r_{i+1}$ in this period, i.e. $Z_j^{(i)} = \mu^{(r_{i+1})}(T_{j+1}) - \mu^{(r_{i+1})}(T_{j})$ 

The key lemma for the third scheme is stated as follows:

\begin{lemma}\label{lem:key3}
	For $i< i^*$ and all $s=0,1,\cdots, \lfloor c\log n\rfloor - 1$
	\begin{equation}
	\E[I_{j}^{(i)}(s)~|~\mathcal{G}_{j-1}] \leq 4\left(\dfrac{\beta_i}{n}\right)^2 + \dfrac{2(\T+s)}{n^{\alpha}} + \dfrac{4(\T+s)}{n} ~~~~~~\text{a.s.}
	\end{equation}
	
	Let $L:=\lfloor 2\left(\frac{c}{\alpha}\log(k-1) + 1 \right) \rfloor$, then for $i\geq i^*$ and all $s=0,1,\cdots, \lfloor c\log n\rfloor - 1$
	\begin{equation}\label{eq67-1}
	\E[I_{j}^{(i)}(s)~|~\mathcal{G}_{j-1}] \leq 4\left(\dfrac{\beta_{i}}{n}\right)^2~~~~~~\text{a.s.}
	\end{equation}
	
\end{lemma}

\begin{proof}
	For the case where $j=0$, the random walkers are set to independent uniform random positions at time $T_j=0$. Using the proof of Lemma \ref{lem:key2} one can deduct a stronger result, i.e.
	\begin{equation}
	\E[I_{0}^{(i)}(s)~|~\mathcal{G}_{-1}] = \E[I_{0}^{(i)}(s)]\leq \left(\dfrac{\beta_i}{n}\right)^2 + \dfrac{2s}{n^{\alpha}} + \dfrac{2s}{n} ~~~~~~\text{a.s.}
	\end{equation}
	for $i < i^*$ and
	\begin{equation}
	\E[I_{0}^{(i)}(s)~|~\mathcal{G}_{-1}]=\E[I_{0}^{(i)}(s)] \leq \left(\dfrac{\beta_{i}}{n}\right)^2~~~~~~\text{a.s.}
	\end{equation}
	for $i\geq i^*$.
	
	For the rest of the proof, we consider $j\geq 1$.
	
	The process after $T_{j-1}$ is conditionally independent of $\mathcal{G}_{j-1}$ given $\mathbf{Q}(T_{j-1}), \mathbf{W}(T_{j-1}), \mathbf{W}(T_{j-1}-1)$. Thus
	
	\begin{equation}\label{k03}
	\E[I_j^{(i)}(s)~|~\mathcal{G}_{j-1}] = \E[I_j^{(i)}(s)~|~\mathbf{Q}(T_{j-1}), \mathbf{W}(T_{j-1}), \mathbf{W}(T_{j-1}-1)]~~~~~~a.s.
	\end{equation}
	
	Define $\mathcal{D}_{j, s}$ to be the event that at least one of $W_1(T_j+s)$ and $W_2(T_j+s)$ was visited at least once between time $T_{j-1}$ and $T_j+s$ (Notice that this is not the same definition as the one in the proof for the second scheme). 
	
	\begin{claim}\label{claim:1}
		\begin{equation}
		\Pr(\mathcal{D}_{j, s}~|~\mathbf{W}(T_{j-1}), \mathbf{W}(T_{j-1}-1)) \leq \dfrac{2(\T+s)}{n^{\alpha}} + \dfrac{4(\T+s)}{n}~~~~~~a.s.
		\end{equation}
	\end{claim}
	
	Now, if $\mathbf{Q}(T_{j-1}) = \mathbf{q}$ is such that $\nu^{(r_i)}(T_j) > \beta_{i}$, then by definition in \eqref{defI2}
	
	\begin{equation}\label{k13}
	\E[I_{j}^{(i)}(s)~|~\mathbf{Q}(T_{j-1}) = \mathbf{q}, \mathbf{W}(T_{j-1}), \mathbf{W}(T_{j-1}-1)] = 0~~~~~~a.s.
	\end{equation} 
	
	For $i< i^*$, if $\mathbf{Q}(T_{j-1}) = \mathbf{q}$ is such that $\nu^{(r_i)}(T_j)\leq \beta_{i}$, then
	
	\begin{equation}\label{k23}
	\begin{split}
	&~~~~\E[I_j^{(i)}(s)~|~\mathbf{Q}(T_{j-1}) = \mathbf{q}, \mathbf{W}(T_{j-1}), \mathbf{W}(T_{j-1}-1)] \\
	&= \Pr(Q_{W_l(T_j+s)}(T_j+s)\geq r_{i+1}-1~\forall l=1,2|~\mathbf{Q}(T_{j-1}) = \mathbf{q}, \mathbf{W}(T_{j-1}), \mathbf{W}(T_{j-1}-1))\\
	&= \Pr(Q_{W_l(T_j+s)}(T_j+s)\geq i~\forall l=1,2|~\mathbf{Q}(T_{j-1}) = \mathbf{q}, \mathbf{W}(T_{j-1}), \mathbf{W}(T_{j-1}-1))\\
	&= \Pr(Q_{W_l(T_j+s)}(T_j+s)\geq i~\forall l=1,2,~\mathcal{D}_{j, s}^c|~\mathbf{Q}(T_{j-1}) = \mathbf{q}, \mathbf{W}(T_{j-1}), \mathbf{W}(T_{j-1}-1)) \\
	&~~~~+ \Pr(Q_{W_l(T_j+s)}(T_j+s)\geq i~\forall l=1,2,~\mathcal{D}_{j, s}|~\mathbf{Q}(T_{j-1}) = \mathbf{q}, \mathbf{W}(T_{j-1}), \mathbf{W}(T_{j-1}-1))\\
	&= \Pr(q_{W_l(T_j+s)}\geq i~\forall l=1,2,~\mathcal{D}_{j, s}^c|~\mathbf{Q}(T_{j-1}) = \mathbf{q}, \mathbf{W}(T_{j-1}), \mathbf{W}(T_{j-1}-1)) \\
	&~~~~+ \Pr(Q_{W_l(T_j+s)}(T_j+s)\geq i~\forall l=1,2,~\mathcal{D}_{j, s}|~\mathbf{Q}(T_{j-1}) = \mathbf{q}, \mathbf{W}(T_{j-1}), \mathbf{W}(T_{j-1}-1))\\
	&\leq \Pr(q_{W_l(T_j+s)}\geq i~\forall l=1,2|~\mathbf{Q}(T_{j-1}) = \mathbf{q}, \mathbf{W}(T_{j-1}), \mathbf{W}(T_{j-1}-1)) \\&~~~~+ \Pr(\mathcal{D}_{j, s}|~\mathbf{Q}(T_{j-1}) = \mathbf{q}, \mathbf{W}(T_{j-1}), \mathbf{W}(T_{j-1}-1))\\
	&= \Pr(q_{W_l(T_j+s)}\geq i~\forall l=1,2~|~ \mathbf{W}(T_{j-1}), \mathbf{W}(T_{j-1}-1)) + \Pr(\mathcal{D}_{j, s}~|~\mathbf{W}(T_{j-1}), \mathbf{W}(T_{j-1}-1))\\
	&= \Pr(q_{W_1(T_j+s)} \geq r_i|~W_1(T_{j-1}), W_2(T_{j-1}-1)) \Pr(q_{W_2(T_j+s)} \geq r_i|~W_2(T_{j-1}), W_2(T_{j-1}-1))\\
	&~~~~ + \Pr(\mathcal{D}_{j, s}~|~\mathbf{W}(T_{j-1}), \mathbf{W}(T_{j-1}-1))\\
	&\leq \left( \beta_i\cdot \dfrac{2}{n}\right)^2  + \dfrac{2(\T+s)}{n^{\alpha}} + \dfrac{4(\T+s)}{n}~~~~~~a.s.
	\end{split}
	\end{equation}
	
	Combining \eqref{k03}\eqref{k13}\eqref{k23} we conclude that 
	\begin{equation}
	\E[I_j^{(i)}(s)~|~\mathcal{G}_{j-1}]\leq 4\left(\dfrac{\beta_i}{n}\right)^2 + \dfrac{2(\T+s)}{n^\alpha} + \dfrac{4(\T+s)}{n}
	\end{equation}
	for $i< i^*$.
	
	To prove the second statement in the Lemma, we notice that it is impossible for a random walker to visit one vertex strictly more than $\dfrac{2\T}{2\lceil \alpha \log_{k-1}n\rceil + 1}+1$ within $2\T$ steps. Therefore for any $0\leq s<\T$, both $W_1(T_j+s)$ and $W_2(T_j+s)$ has been visited no more than
	\begin{equation*}
	2\cdot \left(\dfrac{2\T}{2\lceil \alpha \log_{k-1}n\rceil + 1} + 1\right) < 2\left(\dfrac{c\log(k-1)}{\alpha} + 1\right)
	\end{equation*}
	times within time $T_{j-1}$ to $T_j+s$. Therefore $Q_{W_l(T_j+s)(T_j+s)}\leq q_{W_l(T_j+s)} + L$ a.s. for $l=1,2$ when $\mathbf{Q}(T_{j-1}) = \mathbf{q}$.
	
	For $i\geq i^*$, if $\mathbf{Q}(T_{j-1}) = \mathbf{q}$ is such that $\nu^{(r_i)}(T_j)\leq \beta_{i}$, then
	
	\begin{equation}\label{k33}
	\begin{split}
	&~~~~\E[I_j^{(i)}(s)~|~\mathbf{Q}(T_{j-1}) = \mathbf{q}, \mathbf{W}(T_{j-1}), \mathbf{W}(T_{j-1}-1)] \\
	&= \Pr(Q_{W_l(T_j+s)}(T_j+s)\geq r_{i+1}-1~\forall l=1,2~|~\mathbf{Q}(T_{j-1}) = \mathbf{q}, \mathbf{W}(T_{j-1}), \mathbf{W}(T_{j-1}-1))\\
	&= \Pr(Q_{W_l(T_j+s)}(T_j+s)\geq r_i+L~\forall l=1,2~|~\mathbf{Q}(T_{j-1}) = \mathbf{q}, \mathbf{W}(T_{j-1}), \mathbf{W}(T_{j-1}-1))\\
	&\leq \Pr(q_{W_l(T_j+s)} + L\geq r_i+L~\forall l=1,2~|~\mathbf{Q}(T_{j-1}) = \mathbf{q}, \mathbf{W}(T_{j-1}), \mathbf{W}(T_{j-1}-1))\\
	&\leq \Pr(q_{W_l(T_j+s)} \geq r_i~\forall l=1,2~|~\mathbf{W}(T_{j-1}), \mathbf{W}(T_{j-1}-1))\\
	&\leq  \Pr(q_{W_1(T_j+s)} \geq r_i|~W_1(T_{j-1}), W_2(T_{j-1}-1)) \Pr(q_{W_2(T_j+s)} \geq r_i|~W_1(T_{j-1}), W_2(T_{j-1}-1))\\
	&\leq \left(\beta_{i} \cdot\dfrac{2}{n} \right)^2 
	\end{split}
	\end{equation}
	
	Combining \eqref{k03}\eqref{k13}\eqref{k33} we conclude that
	\begin{equation}
	\E[I_j^{(i)}(s)~|~\mathcal{G}_{j-1}]\leq 4\left(\dfrac{\beta_{i}}{n}\right)^2
	\end{equation}
	for $i\geq i^*$.
	
\end{proof}

\begin{proof}[Proof of Claim \ref{claim:1}]
	\begin{equation*}
	\begin{split}
	&\Pr(\mathcal{D}_{j, s}~|~\mathbf{W}(T_{j-1}), \mathbf{W}(T_{j-1}-1)) \\&\leq \sum_{r=-\T}^{s-1}\Pr(W_1(T_j+r) = W_1(T_j+s)~|~\mathbf{W}(T_{j-1}), \mathbf{W}(T_{j-1}-1)) \\&~~~~+ \sum_{r=-\T}^{s-1}\Pr(W_1(T_j+r) = W_2(T_j+s)~|~\mathbf{W}(T_{j-1}), \mathbf{W}(T_{j-1}-1)) \\&~~~~+ \sum_{r=-\T}^{s-1}\Pr(W_2(T_j+r) = W_1(T_j+s)~|~\mathbf{W}(T_{j-1}), \mathbf{W}(T_{j-1}-1)) \\&~~~~+ \sum_{r=-\T}^{s-1}\Pr(W_2(T_j+r) = W_2(T_j+s)~|~\mathbf{W}(T_{j-1}), \mathbf{W}(T_{j-1}-1))\\
	&=2\sum_{r=-\T}^{s-1}\Pr(W_1(T_j+r) = W_1(T_j+s)~|~\mathbf{W}(T_{j-1}), \mathbf{W}(T_{j-1}-1)) \\&~~~~+ 2\sum_{r=-\T}^{s-1}\Pr(W_1(T_j+r) = W_2(T_j+s)~|~\mathbf{W}(T_{j-1}), \mathbf{W}(T_{j-1}-1))
	\end{split}
	\end{equation*}
	
	From Lemma \ref{lem:alon}, with the assumption on girth we know that
	\begin{equation*}
	\Pr(W_1(T_j+r) = W_1(T_j+s)~|~\mathbf{W}(T_{j-1}), \mathbf{W}(T_{j-1}-1)) \leq \dfrac{1}{n^\alpha}
	\end{equation*}
	
	Given $\mathbf{W}(T_{j-1}), \mathbf{W}(T_{j-1}-1)$, $W_1(T_j+r)$ and $W_2(T_j+s)$ are conditionally independent. Hence
	\begin{equation}
	\begin{split}
	&~~~~\Pr(W_1(T_j+r) = W_2(T_j+s)~|~\mathbf{W}(T_{j-1}), \mathbf{W}(T_{j-1}-1))\\
	&=\sum_{v=1}^n \Pr(W_1(T_j+r) = v~|~W_1(T_{j-1}), W_1(T_{j-1}-1)) \Pr(W_2(T_j+s) = v~|~W_2(T_{j-1}), W_2(T_{j-1}-1))\\
	&\stackrel{\text{(Lemma \ref{lem:mixing})}}{\leq} \sum_{v=1}^n \Pr(W_1(T_j+r) = v~|~\mathbf{W}(T_{j-1}), \mathbf{W}(T_{j-1}-1)) \cdot \dfrac{2}{n}=\dfrac{2}{n}
	\end{split}
	\end{equation}
	
	Combining the above we obtain
	\begin{equation}
	\Pr(\mathcal{D}_{j, s}~|~\mathbf{W}(T_{j-1}), \mathbf{W}(T_{j-1}-1)) \leq \dfrac{2(\T+s)}{n^\alpha} + \dfrac{4(\T+s)}{n}
	\end{equation}
	
\end{proof}

Our application of concentration inequalities will be based on the following observation.

\begin{observation}
	\begin{enumerate}
		\item $\{Z_{2j}\}_{j=0}^\infty$ is an adapted stochastic process w.r.t. the filtration $\{\mathcal{G}_{(2j-1)_+} \}_{j=0}^\infty$
		
		\item $\{Z_{2j+1}\}_{j=0}^\infty$ is an adapted stochastic process w.r.t. the filtration $\{\mathcal{G}_{2j} \}_{j=0}^\infty$.
	\end{enumerate}
\end{observation}

\subsection{Proof of the Main Theorem}

Define $N$ to be the smallest even integer that is greater than $\frac{n}{\lfloor c\log n \rfloor }$. We have $T_N\geq n$. For large $n$ we have
\begin{equation}
N\leq \dfrac{en}{2c\log n}
\end{equation}

Condition on $\mathcal{E}_i$, we have $\sum_{l=0}^{N-1} Z_l^{(i)} $ to be the number of occasions both sampled bins have load at least $r_{i+1} - 1$ before time $T_N\geq n$. Thus $\sum_{l=0}^{N-1} Z_l^{(i)} \geq \mu^{(r_{i+1})}(n)$ holds on $\mathcal{E}_i$.

Let $i^*> 18$ to be determined later. Set $\beta_{18}=\dfrac{n}{6e}$. Then $\nu^{(18)}(n)\leq \beta_{18}$ is always true. Hence $\Pr(\mathcal{E}_{18})=1$.

For $i\geq 18$ we have
\begin{align*}
\Pr(\mathcal{E}_{i+1}^c\cap \mathcal{E}_i) &= \Pr(\nu^{(r_{i+1})}(n) > \beta_{i+1},\mathcal{E}_i)\\
&\leq \Pr(\mu^{(r_{i+1})}(n) > \beta_{i+1},\mathcal{E}_i)\\
&\leq \Pr\left(\sum_{j=0}^{N-1} Z_j^{(i)}> \beta_{i+1},\mathcal{E}_i\right)\\
&\leq \Pr\left(\sum_{j=0}^{N-1} Z_j^{(i)}> \beta_{i+1}\right)\\
&\leq \Pr\left(\sum_{j=0}^{\frac{N}{2}-1} Z_{2j}^{(i)}> \dfrac{\beta_{i+1}}{2}\right) + \Pr\left(\sum_{j=0}^{\frac{N}{2}-1} Z_{2j+1}^{(i)}> \dfrac{\beta_{i+1}}{2}\right)
\end{align*}

Now, for $18\leq i<i^*$, define $\beta_{i+1} = 5en\left(\frac{\beta_i}{n}\right)^2$, where $i^*$ is defined to be the smallest $i$ such that
\begin{equation}
\left(\dfrac{\beta_{i-1}}{n}\right)^2 < \dfrac{9c\log n}{n^\gamma}
\end{equation}
where $\gamma:=\frac{1}{2}\wedge \alpha$. It can be shown that $i^*\leq \log_2\log n + \Theta(1)$.

For $18\leq i<i^*-1$ we have
\begin{equation}
\left(\frac{\beta_i}{n}\right)^2\geq \dfrac{9c\log n}{n^{1\wedge\alpha}}\geq \dfrac{3c\log n}{n^\alpha} + \dfrac{6c\log n}{n}
\end{equation}

Thus for $18\leq i<i^*-1$,
\begin{equation}
\begin{split}
\E[Z_j^{(i)}~|~\mathcal{G}_{j-1}] &\leq \sum_{s=0}^{\lfloor c\log n\rfloor - 1}\left[4\left(\frac{\beta_i}{n}\right)^2 + \dfrac{2(\T+s)}{n^\alpha} + \dfrac{4(\T+s)}{n} \right]\\
&\leq 4c\log n \left(\frac{\beta_i}{n}\right)^2 + \dfrac{3(c\log n)^2}{n^\alpha} + \dfrac{6(c\log n)^2}{n}\\
&\leq 5c\log n \left(\frac{\beta_i}{n}\right)^2=:m_i
\end{split}
\end{equation}

For $18\leq i<i^*-1$ We check that $\beta_{i+1}=2\cdot\frac{en}{2c\log n}\cdot 5c\log n\left(\frac{\beta_i}{n}\right)^2 \geq 2Nm_i$. Hence applying Corollary \ref{berncor} we obtain
\begin{equation}
\begin{split}
\Pr(\mathcal{E}_{i+1}^c \cap \mathcal{E}_i)&\leq \Pr\left(\sum_{j=0}^{\frac{N}{2}-1} Z_{2j}^{(i)}> \dfrac{\beta_{i+1}}{2}\right) + \Pr\left(\sum_{j=0}^{\frac{N}{2}-1} Z_{2j+1}^{(i)}> \dfrac{\beta_{i+1}}{2}\right)\\
&\leq 2\exp\left(-\dfrac{3\cdot\frac{\beta_{i+1}}{2}}{16c\log n}\right)\\
&\leq 2\exp\left(-\dfrac{3\cdot 45ecn^{1-\gamma}\log n}{32c\log n}\right) = \exp(-\dfrac{135e}{32}n^{1-\gamma})\\
&\leq 2\exp(-10n^{1/2})
\end{split}\label{www3}
\end{equation}
for $18\leq i< i^*-1$. 

Now we have
\begin{equation}
\begin{split}
\E[Z_j^{(i^*-1)}~|~\mathcal{G}_{j-1}] &\leq \sum_{s=0}^{\lfloor c\log n\rfloor - 1}\left[4\left(\frac{\beta_{i^*-1}}{n}\right)^2 + \dfrac{2(\T+s)}{n^\alpha} + \dfrac{4(\T+s)}{n} \right]\\
&\leq 4c\log n \left(\frac{\beta_{i^*-1}}{n}\right)^2 + \dfrac{3(c\log n)^2}{n^\alpha} + \dfrac{6(c\log n)^2}{n}\\
&\leq \dfrac{36(c\log n)^2}{n^{\gamma}} + \dfrac{3(c\log n)^2}{n^\alpha} + \dfrac{6(c\log n)^2}{n}\\
&\leq \dfrac{45(c\log n)^2}{n^{\gamma}}=:m_{i^*-1}
\end{split}
\end{equation}

Set $\beta_{i^*} =45ecn^{1-\gamma}\log n$. We have $\beta_{i^*}= 2\cdot \frac{en}{2c\log n}\cdot  \frac{45(c\log n)^2}{n^{\gamma}}\geq 2Nm_{i^*-1}$. By applying Corollary \ref{berncor} again we have
\begin{align}
&~~~~\Pr(\mathcal{E}_{i^*}^c\cap \mathcal{E}_{i^*-1})\leq\Pr\left(\sum_{j=0}^{\frac{N}{2}-1} Z_{2j}^{(i^*-1)}> \dfrac{\beta_{i^*}}{2}\right) + \Pr\left(\sum_{j=0}^{\frac{N}{2}-1} Z_{2j+1}^{(i^*-1)}> \dfrac{\beta_{i^*}}{2}\right) \\
&=2\exp\left(-\dfrac{3\beta_{i^*}}{32c\log n}\right)=\exp\left(-\dfrac{3\cdot 45ecn^{1-\gamma}\log n}{32c\log n}\right) = \exp(-\dfrac{135e}{32}n^{1-\gamma})\\
&\leq \exp(-10n^{1/2})
\label{xxx3}
\end{align}

Now for $i^*\leq i < i^{**}-1$, set $\beta_{i+1} = 4en\left(\frac{\beta_{i}}{n}\right)^2$, where $i^{**}$ is the smallest $i\geq i^*+1$ such that $4en\left(\frac{\beta_{i-1}}{n}\right)^2\leq 16c(\log n)^2$. It can be shown that $i^{**} - i^*\leq \Theta(1)$.

Now, we use \eqref{eq67-1} to bound the conditional expectation for $i\geq i^*$:
\begin{equation}
\E[Z_j^{(i)}~|~\mathcal{G}_{j-1}] \leq 4c\log n\left(\dfrac{\beta_{i}}{n}\right)^2=:m_i
\end{equation}

For $i^*\leq i < i^{**}-1$ we check that 
\begin{equation}
\beta_{i+1} = 2\cdot \dfrac{en}{2}\cdot 4c\log n \left(\dfrac{\beta_{i}}{n}\right)^2 \geq 2Nm_{i}
\end{equation}

Hence for $i^*\leq i < i^{**}-1$ we can apply Corollary \ref{berncor} and obtain
\begin{align}
&~~~~\Pr(\mathcal{E}_{i+1}^c \cap \mathcal{E}_{i}) \\
&\leq  \Pr\left(\sum_{j=0}^{\frac{N}{2}-1} Z_{2j}^{(i)}> \dfrac{\beta_{i+1}}{2}\right) + \Pr\left(\sum_{j=0}^{\frac{N}{2}-1} Z_{2j+1}^{(i)}> \dfrac{\beta_{i+1}}{2}\right)\\
&\leq 2\exp\left(-\dfrac{3\beta_{i+1}}{32c\log n} \right) \leq 2\exp\left(-\dfrac{3\cdot 16c(\log n)^2}{32c\log n} \right)\\
&\leq 2n^{-1.5}
\label{yyy3}
\end{align}

Set $\beta_{i^{**}} = 16c(\log n)^2$. We still have $\beta_{i^{**}}\geq 4en\left(\frac{\beta_{i^{**}-1}}{n}\right)^2\geq 2Nm_{i^{**}-1}$. Applying Corollary \ref{berncor} again we have
\begin{equation}\label{zzz3}
\begin{split}
&~~~~\Pr(\mathcal{E}_{i^{**}}^c \cap \mathcal{E}_{i^{**}-1})\\
&\leq \Pr\left(\sum_{j=0}^{\frac{N}{2}-1} Z_{2j}^{(i^{**}-1)}> \dfrac{\beta_{i^{**}}}{2}\right) + \Pr\left(\sum_{j=0}^{\frac{N}{2}-1} Z_{2j+1}^{(i^{**}-1)}> \dfrac{\beta_{i^{**}}}{2}\right)\\
&\leq 2\exp\left(-\dfrac{3\beta_{i^{**}}}{32c\log n} \right) = 2\exp\left(-\dfrac{3\cdot 16c(\log n)^2}{32c\log n} \right)\\
&\leq 2n^{-1.5}
\end{split}
\end{equation}

Set $\beta_{i^{**}+1} = 0.8$, by Markov Inequality we have
\begin{equation}\label{istarp23}
\begin{split}
&~~~~\Pr(\mathcal{E}_{i^{**}+1}^c \cap \mathcal{E}_{i^{**}})\\
&\leq \Pr\left(\sum_{j=0}^{N-1} Z_j^{(i^{**})} \geq \beta_{i^{**}+1}\right)\leq \dfrac{\E\left[\sum_{j=0}^{N-1} Z_j^{(i^{**})} \right] }{\beta_{i^{**}+1}}\\&\leq \dfrac{Nc\log n \left(\frac{\beta_{i^{**}}}{n}\right)^2 }{\beta_{i^{**}+1}}\leq\dfrac{\frac{en}{2} \left(\frac{\beta_{i^{**}}}{n}\right)^2 }{\beta_{i^{**}+1}}=160e\cdot\dfrac{c^2(\log n)^4}{n}
\end{split}
\end{equation}

Hence, combining \eqref{www3}\eqref{xxx3}\eqref{yyy3}\eqref{zzz3}\eqref{istarp23} we obtain
\begin{align}
\Pr(\mathcal{E}_{i^{**}+1}^c ) &= \Pr(\mathcal{E}_{i^{**}+1}^c\cap \mathcal{E}_9)\leq \Pr\left(\bigcup_{i=9}^{i^{**}}(\mathcal{E}_{i+1}^c\cap \mathcal{E}_i) \right)\leq \sum_{i=9}^{i^{**}}\Pr(\mathcal{E}_{i+1}^c\cap \mathcal{E}_i)\\
&\leq (i^*-9)\cdot 2\exp(-10n^{1/2}) + (i^{**} - i^*)\cdot 2n^{-1.5} + 160e\cdot\dfrac{c^2(\log n)^4}{n}\\
&=o(1)
\end{align}

Therefore, with high probability, $\mathcal{E}_{i^{**}+1}$ is true, that is, no bin has load exceeding $$r_{i^{**} + 1} = i^* + (i^{**} - i^* + 1)(L+1) \leq \log_2 \log n + \Theta(1)$$proving the result.

\section{Discussion}
For the third scheme, while it is not known that the $\Omega(\log n)$ girth assumption is tight in terms of achieving a maximum load of $\log_2\log n + \Theta(1)$ w.h.p. We can show that the girth is required to be at least $\Theta(\frac{\log n}{\log \log n})$ in order to have a maximum load of $\Theta(\log\log n)$ with high probability. More generally, for any $g$, there are expander graphs with girth $g$ such that the maximum load achieved by this graph is $\Omega(\frac{\log n}{g})$ almost surely. We also show that the $\Theta(\frac{\log n}{\log \log n})$ bound for girth is tight, that is, when the girth is $\Omega(\frac{\log n}{\log \log n})$, the maximum load is $\Theta\left(\log\log n\right)$ (where the coefficient for $\log\log n$ is not guaranteed to be smaller than $\frac{1}{\log 2}$). This result gives a complete answer to Alon's open problem.

\begin{thm}
	Let $G$ be a $k$-regular graph on $n$ vertices such that each vertex is contained in a cycle of length $g=g(n)$, then the maximum load yielded by the third scheme is larger than $\Theta(\frac{\log n}{g})$ almost surely.
\end{thm}

\begin{proof}
	The idea of this proof is from \cite{alon2007non}.
	
	First, if $g=\Omega(\log n)$, there is nothing to prove.
	
	Now for the rest of the proof, assume $g = o(\log n)$. Set $b = \lfloor\frac{\log_{k-1} n}{4g}\rfloor$.
	
	\begin{claim}
		Let $V$ be a non-backtracking random walk on $G$, then
		\begin{equation}
		\Pr(V(g+1) = v~|~V(1) = v, V(0) = u) \geq (k-1)^{-g}
		\end{equation}
		for all $u, v\in G$ such that $(u, v)\in E$.
	\end{claim}
	
	\begin{proof}[Proof of Claim]
		Despite the fact that the random walk is non-backtracking, there always exist at least one feasible path of length $g$ from $v$ to itself.
	\end{proof}
	
	Set $c=\frac{1}{2\log(k-1)}$. Then for large $n$, $\frac{1}{4}\log_{k-1} n + 1\leq \lfloor c\log n\rfloor$, which implies that $bg+1\leq \lfloor c\log n\rfloor$.
	
	Define $T_j, \mathcal{G}_j$ in the same way as before.
	Define the events
	\begin{equation}
	\mathcal{C}_j = \{ W_l(T_j) = W_l(T_j+g) = \cdots = W_l(T_j + bg)~~~\forall l=1,2 \}
	\end{equation}
	
	(The above events are analogous to the event $A_j$ in \cite{alon2007non}, which is the event that a single non-backtracking random walker follows the cycle.)
	
	We have
	\begin{equation}
	\Pr(\mathcal{C}_j~|~\mathcal{G}_{j}) \geq (k-1)^{-2gb} \geq n^{-1/2}
	\end{equation}
	
	Let $N = \lfloor \frac{n}{c\log n}-1\rfloor$. We have
	\begin{equation}
	\begin{split}
	\Pr\left(\bigcap_{j=0}^{N-1}\mathcal{C}_j^c\right) &= \Pr(\mathcal{C}_{N-1}^c~|~\bigcap_{j=0}^{N-2}\mathcal{C}_j^c)\cdots \Pr(\mathcal{C}_1|\mathcal{C}_0)\Pr(\mathcal{C}_0)\\
	&\leq (1-n^{-1/2})^N = O\left(\exp\left( -\dfrac{\sqrt{n}}{c \log n}\right)\right)
	\end{split}
	\end{equation}
	
	Hence, one of $\mathcal{C}_j, 0\leq j\leq N-1$, happens almost surely. 
	
	If $\mathcal{C}_j$ is true, then the pair of vertices $W_1(T_j), W_2(T_j)$ are visited simultaneously for at least $b+1$ times. Then, the maximum load of the two bins $W_1(T_j), W_2(T_j)$ is at least $\frac{b+1}{2}$, which implies that the maximum load of all bins is at least $\frac{b+1}{2} = \Theta(\frac{\log n}{g})$.
	
	Therefore, we conclude that the maximum load of all bins is at least $\Theta(\frac{\log n}{g})$ almost surely.
\end{proof}

\begin{thm}
	Let $G$ be a $k$-regular expander graph on $n$ vertices with girth at least $2\lceil\alpha\frac{\log n}{\log \log n}\rceil + 1$, then the maximum load yielded by the third scheme is less than $\kappa\log_2\log n + \Theta(1)$ w.h.p., where $\kappa = \kappa(\alpha, \lambda, k)$ is a positive constant.
\end{thm}

\begin{proof}
	Define $L:=\lfloor 2\left(\frac{c}{\alpha}\log(k-1)\log\log n + 1 \right) \rfloor$
	
	Define $r_i$ as follows:
	\begin{equation}
	r_i=\begin{cases}
	i&i\leq i^*\\
	r_{i-1} + 2^{i - i^*} + 1 &i^*<i\leq i^{**} \\
	r_{i-1} + L + 1 &i>i^{**}
	\end{cases}
	\end{equation}
	where $i^*, i^{**}\in\mathbb{N}$ is determined later. 
	
	(We have $r_{i^**+1} = i^{**} + 2^{i^{**} - i^* + 1} + L-2 $)
	
	Define the events $\mathcal{E}_i:=\{\nu^{(r_i)}(n)\leq \beta_i \}$ as before.
	
	Define
	\begin{equation}
	I_j^{(i)}(s):=\begin{cases}
	1&\text{if }Q_{W_l(T_j+s)}(T_j+s) \geq r_{i+1}- 1~\forall l=1,2\\
	&\text{and }\nu^{(r_i)}(T_j)\leq \beta_i\\
	0&\text{otherwise}
	\end{cases}
	\end{equation}
	
	Let $c>0$ be as specified in the proof of the third scheme. Define the filtration $\mathcal{G}_j$ to be the same as in Section \ref{sec: 3}.
	
	\begin{lemma}\label{lem: 7}
		\begin{enumerate}[(a)]
			\item For $i<i^*$ and $s = 0,1,\cdots,\T - 1$
			\begin{equation}
			\E[I_j^{(i)}~|~\mathcal{G}_{j-1}]\leq 4\left(\dfrac{\beta_i}{n}\right)^2 + 2(\T+s)n^{-\frac{\alpha}{\log\log n}} + 4(\T+s)n^{-1}
			\end{equation}
			
			\item For $i^*\leq i < i^{**}$ and $s = 0,1,\cdots,\T - 1$
			\begin{equation}
			\E[I_j^{(i)}~|~\mathcal{G}_{j-1}]\leq 4\left(\dfrac{\beta_i}{n}\right)^2 + 2 \binom{\T+s}{2^{i-i^*+1}} n^{-2^{i-i^*+1}\frac{\alpha}{\log\log n}} + 4(\T+s)n^{-1}
			\end{equation}
			
		\end{enumerate}
		
	\end{lemma}
	
	\begin{proof}[Proof of Lemma \ref{lem: 7}]
		The proof of part (a) is identical to the proof of Lemma \ref{lem:key3}.
		
		Define $\mathcal{D}_{j, s}^{(l)}$ to be the event that at least one of the vertices $W_1(T_j+s)$ and $W_2(T_j+s)$ was visited at least $l$ times by the random walkers between time $T_{j-1}$ and $T_j+s$. 
		
		Now, for $i^*\leq i < i^{**}$,
		\begin{equation}
		\begin{split}
		&~~~~\E[I_j^{(i)}(s)~|~\mathbf{Q}(T_{j-1}) = \mathbf{q},~\mathbf{W}(T_{j-1}),~\mathbf{W}(T_{j-1}-1)]\\
		&\leq \Pr(Q_{W_l(T_j+s)}(T_j+s)\geq r_{i+1} - 1~\forall l = 1,2~|~\mathbf{Q}(T_{j-1}) = \mathbf{q},~\mathbf{W}(T_{j-1}),~\mathbf{W}(T_{j-1}-1))\\
		&= \Pr(Q_{W_l(T_j+s)}(T_j+s)\geq r_{i} + 2^{i-i^*+1}~\forall l = 1,2~|~\mathbf{Q}(T_{j-1}) = \mathbf{q},~\mathbf{W}(T_{j-1}),~\mathbf{W}(T_{j-1}-1))\\
		&\leq \Pr(Q_{W_l(T_j+s)}(T_j)\geq r_{i} ~\forall l = 1,2~|~\mathbf{Q}(T_{j-1}) = \mathbf{q},~\mathbf{W}(T_{j-1}),~\mathbf{W}(T_{j-1}-1)) \\
		&~~~+ \Pr(\mathcal{D}_{j,s}^{(2^{i-i^*+1})}~|~\mathbf{Q}(T_{j-1}) = \mathbf{q},~\mathbf{W}(T_{j-1}),~\mathbf{W}(T_{j-1}-1)) \\
		&\leq \Pr(q_{W_l(T_j+s)}(T_j)\geq r_{i} ~\forall l = 1,2~|~\mathbf{W}(T_{j-1}),~\mathbf{W}(T_{j-1}-1)) + \Pr(\mathcal{D}_{j,s}^{(2^{i-i^*+1})}~|~\mathbf{W}(T_{j-1}),~\mathbf{W}(T_{j-1}-1)) \\
		&\leq 4\left(\dfrac{\beta_i}{n}\right)^2 + \Pr(\mathcal{D}_{j,s}^{(2^{i-i^*+1})}~|~\mathbf{W}(T_{j-1}),~\mathbf{W}(T_{j-1}-1)) 
		\end{split}
		\end{equation}
		
		We only need to prove the following claim.
		
		\begin{claim}\label{claim:3}
			\begin{equation}
			\Pr(\mathcal{D}_{j,s}^{(l)}~|~\mathbf{W}(T_{j-1}), \mathbf{W}(T_{j-1} -1)) \leq 2\binom{\T+s}{l} n^{-l\frac{\alpha}{\log\log n}} + 4(\T+s)n^{-1}
			\end{equation}
		\end{claim}
		
		\begin{proof}[Proof of Claim \ref{claim:3}]
			\begin{equation}
			\begin{split}
			&~~~~\Pr(\mathcal{D}_{j, s}^{(l)} ~|~\mathbf{W}(T_{j-1}), \mathbf{W}(T_{j-1} -1) )\\
			&\leq \sum_{-\T \leq r_1<\cdots<r_l<s}\Pr(W_1(T_j+r_1) = \cdots = W_1(T_j+r_l) = W_1(T_j+s)~|~\mathbf{W}(T_{j-1}), \mathbf{W}(T_{j-1} -1))\\
			&+\sum_{-\T \leq r_1<\cdots<r_l<s}\Pr(W_2(T_j+r_1) = \cdots = W_2(T_j+r_l) = W_2(T_j+s)~|~\mathbf{W}(T_{j-1}), \mathbf{W}(T_{j-1} -1))\\
			&+\sum_{r=-\T}^{s-1} \Pr(W_1(T_j+r) = W_1(T_j+s)~|~\mathbf{W}(T_{j-1}), \mathbf{W}(T_{j-1} -1))\\
			&+\sum_{r=-\T}^{s-1} \Pr(W_2(T_j+r) = W_2(T_j+s)~|~\mathbf{W}(T_{j-1}), \mathbf{W}(T_{j-1} -1))\\
			&\leq 2 \sum_{-\T \leq r_1<\cdots<r_l<s}\Pr(W_1(T_j+r_1) = \cdots = W_1(T_j+r_l) = W_1(T_j+s)~|~\mathbf{W}(T_{j-1}), \mathbf{W}(T_{j-1} -1))\\
			&+2\sum_{r=-\T}^{s-1} \Pr(W_1(T_j+r) = W_1(T_j+s)~|~\mathbf{W}(T_{j-1}), \mathbf{W}(T_{j-1} -1))
			\end{split}
			\end{equation}
			
			Using the same proof for Claim \ref{claim:1}, one can obtain
			\begin{equation}
			2\sum_{r=-\T}^{s-1} \Pr(W_1(T_j+r) = W_1(T_j+s)~|~\mathbf{W}(T_{j-1}), \mathbf{W}(T_{j-1} -1)) \leq \dfrac{4(\T+s)}{n}
			\end{equation}
			
			By the girth assumption and Lemma \ref{lem:alon}, we have
			\begin{equation}
			\Pr(W_1(T_j+r_1) = W_1(T_j+r_2)~|~\mathbf{W}(T_{j}+r_1), \mathbf{W}(T_{j}+r_1-1))\leq n^{-\frac{\alpha}{\log\log n}}
			\end{equation}
			
			Hence we have
			\begin{equation}
			\Pr(W_1(T_j+r_1) = \cdots = W_1(T_j+r_l) = W_1(T_j+s)~|~\mathbf{W}(T_{j-1}), \mathbf{W}(T_{j-1} -1))\leq (n^{-\frac{\alpha}{\log\log n}})^l
			\end{equation}
			
			We conclude that
			\begin{equation}
			\Pr(\mathcal{D}_{j, s}^{(l)} ~|~\mathbf{W}(T_{j-1}), \mathbf{W}(T_{j-1} -1) ) \leq 2\binom{\T+s}{l} n^{-l\frac{\alpha}{\log\log n}} +  \dfrac{4(\T+s)}{n}
			\end{equation}
		\end{proof}
		
		Given the claim, the rest of the proof is identical to that of Lemma \ref{lem:key3}.
	\end{proof}
	
	Let $N$ be the smallest even integer greater than $\frac{n}{\lfloor c\log n\rfloor}$. For large $n$ we have $N\leq \frac{en}{2c\log n}$. We have $T_N\geq n$. Thus $\sum_{j=0}^{N-1} Z_l^{(j)} \geq \mu^{(r_i+1)}(n)$ holds on $\mathcal{E}_i$.
	
	Let $i^*>18$ be determined later. Set $\beta_{18} = \dfrac{n}{6e}$. Then $\nu^{(18)}(n)\leq \beta_{18}$ is always true. Hence $\Pr(\mathcal{E}_{18}) = 1$.
	
	For $i\geq 18$ we have
	\begin{equation}
	\begin{split}
	\Pr(\mathcal{E}_{i+1}^c\cap \mathcal{E}_i) \leq \Pr\left(\sum_{j=0}^{\frac{N}{2}-1} Z_{2j}^{(i)}> \dfrac{\beta_{i+1}}{2}\right) + \Pr\left(\sum_{j=0}^{\frac{N}{2}-1} Z_{2j+1}^{(i)}> \dfrac{\beta_{i+1}}{2}\right)
	\end{split}
	\end{equation}
	as before in the proof for Theorem \ref{thm:3works}.
	
	Now, for $18 \leq i <i^*-1$, set $\beta_{i+1} = 5en\left(\frac{\beta_i}{n} \right)^2$, where $i^*$ is the smallest $i$ such that
	\begin{equation}
	\left(\dfrac{\beta_{i-1}}{n} \right)^2 < (9c\log n)n^{-\frac{\alpha}{2\log\log n}}
	\end{equation}
	It can be shown that $i^*\leq \log_2\log n + \Theta(1)$.
	
	For $18\leq i<i^*-1$, let $n$ be such that $\log\log n \geq \alpha$, we have
	\begin{equation}
	\left(\dfrac{\beta_i}{n}\right)^2 \geq 9(c\log n) n^{-\frac{\alpha}{2\log\log n}} \geq 3(c\log n)n^{-\frac{\alpha}{\log\log n}} + 6(c\log n)n^{-1}
	\end{equation}
	
	Thus for $18\leq i<i^*-1$, 
	\begin{equation}
	\begin{split}
	\E[Z_j^{(i)}~|~\mathcal{G}_{j-1}] &\leq \sum_{s=0}^{\T-1} \left[4\left(\dfrac{\beta_i}{n}\right)^2 + 2(\T+s)n^{-\frac{\alpha}{\log\log n}} + 4(\T+s)n^{-1} \right]\\
	&\leq 4c\log n\left(\dfrac{\beta_i}{n}\right)^2 + 3(c\log n)n^{-\frac{\alpha}{\log\log n}} + 6(c\log n)n^{-1}\\
	&\leq 5c\log n\left(\dfrac{\beta_i}{n}\right)^2=:m_i
	\end{split}
	\end{equation}
	
	For $18\leq i < i^*-1$ we check that $\beta_{i+1}\geq 2Nm_i$. Hence applying Corollary \ref{berncor} we obtain
	\begin{equation}
	\begin{split}
	\Pr(\mathcal{E}_{i+1}^c \cap \mathcal{E}_i)\leq 2\exp\left(-\dfrac{3(\beta_{i+1}/2)}{16c\log n} \right) \leq 2\exp\left(-\dfrac{135e}{32}n^{1-\frac{\alpha}{2\log\log n}}\right)
	\end{split}
	\end{equation}
	
	Now we have
	\begin{equation}
	\begin{split}
	\E[Z_j^{(i^*-1)}~|~\mathcal{G}_{j-1}] &\leq \sum_{s=0}^{\T-1} \left[4\left(\dfrac{\beta_{i^*-1}}{n}\right)^2 + 2(\T+s)n^{-\frac{\alpha}{\log\log n}  } + 4(\T+s) n^{-1}\right]\\
	&\leq 4c\log n \left(\dfrac{\beta_{i^*-1}}{n}\right)^2 + 3(c\log n)^2 n^{-\frac{\alpha}{\log\log n} } + 6(c\log n)^2 n^{-1}\\
	&\leq 36(c\log n)^2 n^{-\frac{\alpha}{2\log\log n}} + 3(c\log n)^2 n^{-\frac{\alpha}{\log\log n} } + 6(c\log n)^2 n^{-1}\\
	&\leq 45(c\log n)^2 n^{-\frac{\alpha}{2\log\log n}} =:m_{i^*-1}
	\end{split}
	\end{equation}
	
	Set $\beta_{i^*} = 45e(c\log n)n^{1-\frac{\alpha}{2\log\log n}}$. We have $\beta_{i^*} \geq 2Nm_{i^*-1}$. Applying Corollary \ref{berncor} again we have
	\begin{equation}
	\begin{split}
	\Pr(\mathcal{E}_{i^*}^c\cap \mathcal{E}_{i^*-1}) \leq 2\exp\left(-\dfrac{135e}{32}n^{1-\frac{\alpha}{2\log\log n}}\right)
	\end{split}
	\end{equation}
	
	Now, for $i^*\leq i < i^{**}-1$, set $\beta_{i+1} = 6en\left(\frac{\beta_i}{n}\right)^2$, where $i^{**}$ is the smallest $i\geq i^*+1$ such that
	\begin{equation}
	6en\left(\frac{\beta_{i-1}}{n}\right)^2 \leq 16c(\log n)^2
	\end{equation}
	(As a result, $\beta_i \geq 16c(\log n)^2$ for $i^*\leq i < i^{**}$)
	
	It can be shown that $i^{**} - i^* \leq \log_2\log \log n + \Theta(1)$. Therefore
	\begin{equation}
	(2\T)^{2^{i^{**} - i^*+1}} = O(\exp((\log\log n)^2)) = o\left(\exp\left(\dfrac{\alpha\log n}{\log\log n}\right)\right)= o(n^{\frac{\alpha}{\log\log n}})
	\end{equation}
	
	Therefore, there exist $\overline{n}\in\mathbb{N}$ be such that
	\begin{equation}\label{hbindbase}
	\left(\dfrac{\beta_{i^*}}{n}\right)^2 = (45e)^2(c\log n)^2 n^{-\frac{\alpha}{\log\log n}}\geq 2\cdot (2\T)^{2^{i^{**} - i^*+1}} n^{-\frac{2\alpha}{\log\log n}}
	\end{equation}
	for all $n\geq \overline{n}$.
	
	\begin{claim}
		$\left(\dfrac{\beta_{i}}{n}\right)^2 \geq 2\cdot (2\T)^{2^{i^{**} - i^*+1}} n^{-2^{i-i^*+1}\frac{\alpha}{\log\log n}}$ holds for all $i^*\leq i < i^{**}$ and $n\geq \overline{n}$.
	\end{claim}
	
	\begin{proof}[Proof of Claim]
		Proof by induction. We have already established the induction base in \eqref{hbindbase}.
		
		Now, suppose the result holds for $i$. we then have
		\begin{equation}
		\beta_{i+1} = 6en\left(\dfrac{\beta_{i}}{n}\right)^2 \geq 12en\cdot (2\T)^{2^{i^{**} - i^*+1}} n^{-2^{i-i^*+1}\frac{\alpha}{\log\log n}}
		\end{equation}
		hence
		\begin{align*}
		\left(\frac{\beta_{i+1}}{n}\right)^2 &\geq \left(12e\cdot (2\T)^{2^{i^{**} - i^*+1}} n^{-2^{i-i^*+1}\frac{\alpha}{\log\log n}}\right)^2\\
		&=\left(12e\cdot (2\T)^{2^{i^{**} - i^*+1}}\right)^2 n^{-2^{i-i^*+2}\frac{\alpha}{\log\log n}}\\
		&\geq  2\cdot (2\T)^{2^{i^{**} - i^*+1}} n^{-2^{(i+1)-i^*+1}\frac{\alpha}{\log\log n}}
		\end{align*}
		establishing the induction step.
	\end{proof}
	
	As a result of the claim, we have
	\begin{equation}
	\left(\dfrac{\beta_{i}}{n}\right)^2 \geq 2\cdot \binom{\T+s}{2^{i-i^*+1}}  n^{-2^{i-i^*+1}\frac{\alpha}{\log\log n}}
	\end{equation}
	for all $i^*\leq i < i^{**}$ and $s = 0 ,1,\cdots, \T-1$ for $n\geq \overline{n}$.
	
	Assume that $\overline{n}$ is sufficiently large, we also have
	\begin{equation}
	\left(\dfrac{\beta_{i}}{n}\right)^2 \geq \dfrac{16c(\log n)^2}{6en} \geq  \dfrac{4(\T+s)}{n}
	\end{equation}
	for $i^*\leq i < i^{**}-1$ and $s = 0 ,1,\cdots, \T-1$ for $n\geq \overline{n}$.
	
	Therefore, for $i^*\leq i<i^{**}-1$, we have
	\begin{equation}
	\begin{split}
	\E[Z_j^{(i)}~|~\mathcal{G}_{j-1}]&\leq \sum_{s=0}^{\T-1} \left[4\left(\dfrac{\beta_{i}}{n}\right)^2 + 2 \binom{\T+s}{2^{i-i^*+1}}  n^{-2^{i-i^*+1}\frac{\alpha}{\log\log n}} + \dfrac{4(\T+s)}{n} \right]\\
	&\leq \sum_{s=0}^{\T-1} 6\left(\dfrac{\beta_{i}}{n}\right)^2 \leq 6(c\log n) \left(\dfrac{\beta_{i}}{n}\right)^2 =:m_i
	\end{split}
	\end{equation}
	
	We check that $\beta_{i+1}\geq 2Nm_i$. Applying Corollary \ref{berncor} again we have
	\begin{equation}
	\begin{split}
	\Pr(\mathcal{E}_{i+1}^c\cap \mathcal{E}_i) &\leq 2\exp\left(-\dfrac{3(\beta_{i+1}/2)}{16c\log n}\right) \leq 2\exp\left(-\dfrac{24c(\log n)^2}{16c\log n}\right)\\ 
	&= 2n^{-1.5}
	\end{split}
	\end{equation}
	for $i^* \leq i < i^{**} - 1$.
	
	Define $\beta_{i^{**}} = 16c(\log n)^2 \geq 6en\left(\frac{\beta_{i^{**}-1}}{n} \right)^2$ we have
	\begin{equation}
	\begin{split}
	\E[Z_j^{(i^{**} - 1)}~|~\mathcal{G}_{j-1}]&\leq \sum_{s=0}^{\T-1} \left[4\left(\dfrac{\beta_{i^{**} - 1}}{n}\right)^2 + 2 \binom{\T+s}{2^{i^{**} -i^*}}  n^{-2^{i^{**}-i^*}\frac{\alpha}{\log\log n}} + \dfrac{4(\T+s)}{n} \right]\\
	&\leq \sum_{s=0}^{\T-1} \left[5\left(\dfrac{\beta_{i^{**} - 1}}{n}\right)^2 + \dfrac{4(\T+s)}{n} \right]\\
	&\leq \sum_{s=0}^{\T-1} \left[5\cdot \dfrac{16c(\log n)^2}{6en}  + \dfrac{4(\T+s)}{n} \right]\\
	&\leq \sum_{s=0}^{\T-1} \left[6\cdot \dfrac{16c(\log n)^2}{6en}  \right]~~~~~~\text{(for $n\geq \overline{n}$)}\\
	&\leq (c\log n)(16ec(\log n)^2) =: m_{i^{**}-1}
	\end{split}
	\end{equation}
	
	We check that $\beta_{i^{**}}\geq 2Nm_{i^{**}-1}$. Hence we apply Corollary \ref{berncor} again and obtain
	\begin{equation}
	\begin{split}
	\Pr(\mathcal{E}_{i^{**}}^c\cap \mathcal{E}_{i^{**}-1}) &\leq 2\exp\left(-\dfrac{3(\beta_{i^{**}}/2)}{16c\log n}\right) = 2\exp\left(-\dfrac{24c(\log n)^2}{16c\log n}\right)\\ 
	&= 2n^{-1.5}
	\end{split}
	\end{equation}
	
	Set $\beta_{i^{**} + 1} = 0.8$, , by Markov Inequality we have
	\begin{equation}
	\begin{split}
	&~~~~\Pr(\mathcal{E}_{i^{**}+1}^c \cap \mathcal{E}_{i^{**}})\\
	&\leq \Pr\left(\sum_{j=0}^{N-1} Z_j^{(i^{**})} \geq \beta_{i^{**}+1}\right)\leq \dfrac{\E\left[\sum_{j=0}^{N-1} Z_j^{(i^{**})} \right] }{\beta_{i^{**}+1}}\\&\leq \dfrac{Nc\log n \left(\frac{\beta_{i^{**}}}{n}\right)^2 }{\beta_{i^{**}+1}}\leq\dfrac{\frac{en}{2} \left(\frac{\beta_{i^{**}}}{n}\right)^2 }{\beta_{i^{**}+1}}=160e\cdot\dfrac{c^2(\log n)^4}{n}
	\end{split}
	\end{equation}
	
	Now we conclude that
	\begin{align}
	\Pr(\mathcal{E}_{i^{**}+1}^c ) &= \Pr(\mathcal{E}_{i^{**}+1}^c\cap \mathcal{E}_{18})\leq \Pr\left(\bigcup_{i=18}^{i^{**}}(\mathcal{E}_{i+1}^c\cap \mathcal{E}_i) \right)\leq \sum_{i=18}^{i^{**}}\Pr(\mathcal{E}_{i+1}^c\cap \mathcal{E}_i)\\
	&\leq (i^*-18)\cdot 2\exp\left(-\dfrac{135e}{32}n^{1-\frac{\alpha}{2\log\log n}}\right) + (i^{**} - i^*)\cdot 2n^{-1.5} + 160e\cdot\dfrac{c^2(\log n)^4}{n}\\
	&=o(1)
	\end{align}
	
	Therefore, with high probability, $\mathcal{E}_{i^{**}+1}$ is true, that is, no bin has load exceeding 
	\begin{equation}
	r_{i^**+1} = i^{**} + 2^{i^{**} - i^* + 1} + L-2 = \Theta(\log_2\log n)
	\end{equation}
	proving the result.
\end{proof}

\section{Conclusion}

In this paper we proposed and analyzed three schemes for ball-in-bins allocation. All three schemes are variants of the \emph{power-of-$d$} scheme, with bins sampled through $d$ random walk processes on some underlying graph. We show that both schemes can yield the same performance as power-of-$d$, that is, the maximum load is bounded by $\log\log n/\log d+\Theta(1)$ with high probability. Both schemes can be considered as a \emph{derandomized}, or \emph{pseudo-random} version of \emph{power-of-$d$} scheme.

This paper opens several future works. First, matching lower bounds for the proposed two schemes are needed. Secondly, as it is well known that when $d=\Theta(\log n)$, the maximum load is $\Theta(1)$ with high probability, it is natural to ask the question whether the maximum load will be $\Theta(1)$ with high probability when $d=\Theta(\log n)$ in both proposed schemes. Finally, the analysis of performance of the schemes in queuing system settings are needed.

\bibliographystyle{plain}
\bibliography{mybib}

\newpage
\appendix
\section{Appendix}

\begin{proof}[Proof of Lemma \ref{lem:nreset}]
	Let $p:=\Pr(T_1 = \lfloor c\log n\rfloor )$. Since $T_1\leq \lfloor c\log n\rfloor$ almost surely, condition (\ref{nllrst}) implies that $p\geq 0.9$.
	
	Let $\tau_j:=T_{j+1} - T_j$ for $j=0,1,2,\cdots, N-1$. $\tau_j$ are i.i.d. with $\tau_1\equiv T_1$. We have
	\begin{equation}
	\begin{split}
	\E[\tau_j]=\E[T_1] \geq p\lfloor c\log n\rfloor
	\end{split}
	\end{equation}
	and
	\begin{equation}
	\begin{split}
	\Var[\tau_j] &= \E[\tau_j^2] - (\E[\tau_j])^2\\
	&\leq (\lfloor c\log n \rfloor)^2 - p^2 (\lfloor c\log n \rfloor)^2\\
	&=(1-p^2) (\lfloor c\log n \rfloor)^2
	\end{split}
	\end{equation}
	
	We have
	\begin{equation}\label{21}
	\begin{split}
	\Pr(T_N \leq n) &= \Pr\left(\sum_{j=0}^{N-1} \tau_j \leq n \right)\\
	&=\Pr\left(\sum_{j=0}^{N-1} (\tau_j - \E[\tau_j]) \leq -(N\E[\tau_1] -n) \right)\\
	&\leq \Pr\left(\sum_{j=0}^{N-1} (\tau_j - \E[\tau_j]) \leq -(Np\lfloor c\log n\rfloor -n) \right)
	\end{split}
	\end{equation}
	
	For $n\geq \exp(\frac{10}{c})$, we have $0.1c\log n \geq 1$, hence
	\begin{equation}
	\lfloor c\log n\rfloor \geq c\log n - 1 \geq 0.9 c\log n
	\end{equation}
	
	Let $0<\alpha< 1$. For $n\geq 3 \vee \left(\dfrac{40c}{\alpha e}\right)^{1/(1-\alpha)}$, we have $\dfrac{n}{\log n}\geq n^{1-\alpha} \cdot \dfrac{e^{\alpha \log n}}{\log n} \geq \alpha n^{1-\alpha} \geq \dfrac{40 c}{e}$, hence
	\begin{equation}
	N \geq \dfrac{en}{2c\log n}-1\geq \dfrac{en}{2c\log n} - \dfrac{en}{40c\log n}=\dfrac{19en}{40c\log n}
	\end{equation}
	
	Hence under the above condition for $n$, we have
	\begin{equation}
	Np\lfloor c\log n\rfloor - n \geq 0.9\cdot\dfrac{171e n}{400} - n\geq \dfrac{n}{25}
	\end{equation}
	
	Now, apply Bernstein Inequality we have
	\begin{equation}\label{22}
	\begin{split}
	&~~~~\Pr\left(\sum_{j=0}^{N-1} (\tau_j - \E[\tau_j]) \leq -(Np\lfloor c\log n\rfloor -n) \right)\\
	&\leq \Pr\left(\sum_{j=0}^{N-1} (\tau_j - \E[\tau_j]) \leq -\dfrac{n}{25} \right)\\
	&\leq \exp\left(-\dfrac{(n/25)^2}{2N(1-p^2)(\lfloor c\log n\rfloor)^2 + \frac{2}{3}\lfloor c\log n\rfloor (n/25)} \right)\\
	&\leq \exp\left(-\dfrac{(n/25)^2}{\dfrac{en}{c\log n}(1-p^2)( c\log n)^2 + \frac{2}{3} c\log n\cdot (n/25)} \right)\\
	&=\exp\left(-\dfrac{n^2/625}{e(1-p^2)nc\log n + \frac{2}{75} nc\log n\cdot } \right)\\
	&=\exp\left(-\dfrac{1}{625\left[e(1-p^2) + \dfrac{2}{75}\right]}\cdot \dfrac{n}{c\log n} \right)\\
	&\leq \exp\left(-\dfrac{n}{340c\log n}\right)
	\end{split}
	\end{equation}
	
	Hence we conclude from \eqref{21}\eqref{22} that $\Pr(A^c) \leq \exp\left(-\dfrac{n}{340c\log n}\right)$.
\end{proof}

\begin{proof}[Proof of Lemma \ref{lem:mixing}]
	The mixing time result we need is slightly different from the one that Alon et al. proved. We also need a more precise result. Hence we provide the proof here.
	
	Let $A_{u,v}^{(t)}$ denote the number of non-backtracking walks of length $t$ from $u$ to $v$. Then we have the recurrence relation on matrices
	\begin{align*}
	A^{(1)} &= A\\
	A^{(2)} &= A^2 - dI\\
	A^{(t+1)}&= AA^{(t)} - (d-1)A^{(t-1)}~~~\forall t\geq 2
	\end{align*}
	
	Define
	\begin{equation}
	\tilde{P}^{(t)}=\dfrac{A^{(t)}}{k(k-1)^{t-1}}
	\end{equation}
	
	$\tilde{P}^{(t)}$ is the $t$-step transition probability matrix of a non-backtracking random walk, where the first step of the walk is not required to avoid any neighbor of the starting vertex. Let $\mu_1(t) = 1, \mu_2(t),\cdots, \mu_n(t)$ denote the eigenvalues of $\tilde{P}^{(t)}$ and $\mu(t):=\max\{|\mu_2(t)|,\cdots,|\mu_n(t)| \}$. 
	
	Alon et al. proved in their Claim 2.2 that
	\begin{equation}\label{expander0}
	\max_{u,v}\left|\tilde{P}_{uv}^{(t)}-\dfrac{1}{n} \right|\leq \mu(t)
	\end{equation}
	for every $n$ and $t$.
	
	Let $k=\lambda_1^{(n)}\geq \lambda_2^{(n)}\geq \cdots\geq \lambda_n^{(n)}$ be the eigenvalues of the adjacency matrix of the expander graph $G^{(n)}$. Let $\lambda$ be as defined in Definition \ref{def:expander}. Pick $\tilde{\lambda}\in (\lambda, k)$. We have
	\begin{equation}\label{expander1}
	\max\{|\lambda_2^{(n)}|,\cdots, |\lambda_n^{(n)}|\} \leq \tilde{\lambda}
	\end{equation}
	for all large $n$.
	
	Alon et al. have proved that
	\begin{equation}\label{expander2}
	\mu_i(t) = \dfrac{1}{\sqrt{k(k-1)^{t-1}}}q_t\left(\dfrac{\lambda_i^{(n)}}{2\sqrt{k-1}} \right)
	\end{equation}
	
	where $q_t(x)$ is a polynomial satisfying
	\begin{align*}
	q_{1}(x) &= \sqrt{\dfrac{k-1}{k}} \\
	q_{2}(x)& = \sqrt{\dfrac{k-1}{k}}(4x^2-1) - \dfrac{1}{\sqrt{k(k-1)}}\\
	q_{t+1}(x)&=2xq_t(x) - q_{t-1}(x)~~~~~~\forall t\geq 2
	\end{align*}
	
	It can also be shown (through induction) that, if $U_t(x)$ is the Chebyshev polynomial of the second kind, i.e.
	\begin{align*}
	U_t(\cos\theta) = \dfrac{\sin((t+1)\theta)}{\sin(\theta)}
	\end{align*}
	then
	\begin{align}\label{qk2ndform}
	q_t(x) = \sqrt{\dfrac{k-1}{k}}U_t(x) - \dfrac{1}{\sqrt{k(k-1)}}U_{t-2}(x)
	\end{align}
	
	\begin{claim}\label{claim:0}
		The polynomials $q_t(x)$ can be bounded by
		\begin{equation}\label{expander3}
		q_t(x)\leq \sqrt{\dfrac{k-1}{k}}(t+1)[\psi(|x|)]^{t} + \dfrac{1}{\sqrt{k(k-1)}}(t-1)[\psi(|x|)]^{t-2}
		\end{equation}
		where $\psi:\mathbb{R}_+\mapsto \mathbb{R}_+$ is defined as
		\begin{equation}
		\psi(x)=\begin{cases}
		1&x\leq 1\\
		x+\sqrt{x^2-1}&x\geq 1
		\end{cases}
		\end{equation}
	\end{claim}
	
	Given the claim, observing that the right-hand-side of (\ref{expander3}) is an increasing function of $|x|$, combining \eqref{expander1}\eqref{expander2}\eqref{expander3} we have
	\begin{align*}
	\mu(t)&\leq \dfrac{1}{\sqrt{k(k-1)^{t-1}}}\left( \sqrt{\dfrac{k-1}{k}}(t+1)\left[\psi\left(\dfrac{\tilde{\lambda}}{2\sqrt{k-1}} \right)\right]^{t} + \dfrac{1}{\sqrt{k(k-1)}}(t-1)\left[\psi\left(\dfrac{\tilde{\lambda}}{2\sqrt{k-1}} \right)\right]^{t-2} \right)\\
	&=\dfrac{k-1}{k}(t+1)\beta^t + \dfrac{1}{k(k-1)}(t-1)\beta^{t-2}
	\end{align*}
	where
	\begin{align*}
	\beta := \dfrac{1}{\sqrt{k-1}}\psi\left(\dfrac{\tilde{\lambda}}{2\sqrt{k-1}} \right) < \dfrac{1}{\sqrt{k-1}}\psi\left(\dfrac{k}{2\sqrt{k-1}} \right)=1
	\end{align*}
	
	Thus we obtain a uniform estimate
	\begin{equation}
	\max_{u,v}\left|\tilde{P}_{uv}^{(t)}-\dfrac{1}{n} \right|\leq \dfrac{k-1}{k}(t+1)\beta^t + \dfrac{1}{k(k-1)}(t-1)\beta^{t-2}
	\end{equation}
	observe that the right-hand-side is monotonically decreasing in $t$ for large $t$.
	
	Pick $c = -\dfrac{2}{\log \beta} > 0$ and set $\tau = \lfloor c\log n\rfloor$, for sufficiently large $n$ we have
	\begin{equation}
	\max_{u,v}\left|\tilde{P}_{uv}^{(t)}-\dfrac{1}{n} \right|\leq \dfrac{k-1}{k}(\tau+1)\beta^\tau + \dfrac{1}{k(k-1)}(\tau-1)\beta^{\tau-2}=\Theta\left(\dfrac{\log n}{n^2}\right)~~~~~~\forall t\geq \tau
	\end{equation}
	
	Hence, for sufficiently large $n$, we have
	\begin{equation}\label{expander99}
	\max_{u,v}\left|\tilde{P}_{uv}^{(t)}-\dfrac{1}{n} \right|\leq \left(1-\dfrac{2}{k}\right)\dfrac{1}{n}~~~~~~\forall t\geq \tau
	\end{equation}
	
	The last step is to bound $\tilde{P}_{u_1, v, u_0}^{(t)}:=\Pr(V^{(n)}(t+1)=v~|V^{(n)}(0) = u_0, V^{(n)}(1) = u_1)$ via $\tilde{P}_{u_1, v}^{(t)}$: Let $A_{u_1, v, u_0}^{(t)}$ be the number of non-backtracking walks of length $t$ from $u_1$ to $v$ such that the second visited vertex is not $u_0$. We have
	\begin{equation}
	\tilde{P}_{u_1, v, u_0}^{(t)} = \dfrac{A_{u_1, v, u_0}^{(t)}}{(k-1)^t}\leq \dfrac{A_{u_1, v}^{(t)}}{(k-1)^{t}}=\dfrac{k}{k-1}\tilde{P}_{u_1, v}^{(t)}
	\end{equation}
	
	Thus
	\begin{equation}\label{expander100}
	\tilde{P}_{u_1,v}^{(t)} \leq \left(2-\dfrac{2}{k}\right)\dfrac{1}{n}~~~~~~\Rightarrow~~~~~~\tilde{P}_{u_1, v, u_0}^{(t)}\leq \dfrac{2}{n}
	\end{equation}
	
	Combining \eqref{expander99} and \eqref{expander100} we prove the result.
	
	\begin{proof}[Proof of Claim \ref{claim:0}]
		The proof is similar to the proof of Lemma 2.3 by Alon et al.
		
		Using the second form of $q_k$ in \eqref{qk2ndform}, we have
		\begin{align*}
		q_t(x) = \begin{cases}
		\sqrt{\dfrac{k-1}{k}}\dfrac{\sin((t+1)\theta)}{\sin(\theta)} - \dfrac{1}{\sqrt{k(k-1)}} \dfrac{\sin((t-1)\theta)}{\sin(\theta)}&|x|\leq 1\\
		\mathrm{sgn}(x)^t\left[\sqrt{\dfrac{k-1}{k}}\dfrac{\sinh((t+1)\theta)}{\sinh(\theta)} - \dfrac{1}{\sqrt{k(k-1)}} \dfrac{\sinh((t-1)\theta)}{\sinh(\theta)}\right]&|x|>1
		\end{cases} 
		\end{align*}
		where
		\begin{itemize}
			\item $\theta = \arccos(x)$ if $|x|\leq 1$
			\item $\theta = \mathrm{arccosh}(|x|):=\log(|x|+\sqrt{x^2-1})$ if $|x|>1$.
		\end{itemize}
		
		It can be shown through induction (by expanding $\sin(x+y)$) that
		\begin{align*}
		|\sin(t\theta)|\leq t|\sin\theta|~~~~~~\forall t\in\mathbb{N}
		\end{align*}
		
		We also have for $\theta > 0$
		\begin{align*}
		\left|\dfrac{\sinh(t\theta)}{\sinh(\theta)}\right| &= \left|\dfrac{e^{t\theta} -e^{-t\theta}  }{e^{\theta} -e^{-\theta}}\right| = e^{-(t-1)\theta} \left|\dfrac{e^{2t\theta} -1  }{e^{2\theta} -1}\right|\\
		&=e^{-(t-1)\theta}\left|\sum_{l=0}^{t-1} e^{2l\theta}\right|\leq e^{-(t-1)\theta}\cdot te^{2(t-1)\theta}\\
		&=t\cdot e^{(t-1)\theta} = t\cdot \psi(|x|)^{t-1}~~~~~~\forall t\in\mathbb{N}
		\end{align*}
		
		Thus we conclude that
		\begin{equation}
		|q_t(x)| \leq \begin{cases}
		\sqrt{\dfrac{k-1}{k}}(t+1) + \dfrac{1}{\sqrt{k(k-1)}}(t-1)&|x|\leq 1\\
		\sqrt{\dfrac{k-1}{k}}(t+1)[\psi(|x|)]^{t} + \dfrac{1}{\sqrt{k(k-1)}}(t-1)[\psi(|x|)]^{t-2}&|x|>1
		\end{cases} 
		\end{equation}
		
	\end{proof}
\end{proof}

\end{document}